\documentclass[oneside]{amsart}

\usepackage[utf8]{inputenc}
\usepackage{amsfonts}
\usepackage{amssymb}
\usepackage{amsthm}
\usepackage{amsmath}
\usepackage{mathtools}
\usepackage{epsfig}   
\usepackage{epic,eepic}      
\usepackage[english]{babel}
\usepackage{amsopn}
\usepackage{mathrsfs} 
\usepackage[style=alphabetic,firstinits=true,  maxnames=10,  sortcites=true, doi=false, sorting=nyt, backend=bibtex]{biblatex}
\usepackage{tikz}
\usepackage{verbatim}
\usetikzlibrary{decorations.markings,arrows}

\usepackage{marginnote}
\usepackage{lettrine}
\usepackage{graphicx}
\usepackage{booktabs}
\usepackage{tikz}
\usepackage{enumitem}
\usepackage[babel]{csquotes}
\usepackage{guit}
\usepackage{booktabs}
\usepackage{multirow}

\usepackage{xypic}
\usepackage[colorinlistoftodos]{todonotes}
\usepackage{tikz}
\usetikzlibrary{arrows}
\usepackage{hyperref}
\hypersetup{
    colorlinks=false,
    pdfborder={0 0 0},
}
\theoremstyle{theorem}
\newcounter{dummy} \numberwithin{dummy}{section}
\newtheorem{thm}[dummy]{Theorem}
\newtheorem{defn}[dummy]{Definition}

\newtheorem{cor}[dummy]{Corollary}
\newtheorem{prop}[dummy]{Proposition}
\newtheorem{claim}[dummy]{Claim}

\theoremstyle{remark}
\newtheorem{rmk}[dummy]{Remark}
\newtheorem{exam}[dummy]{Example}

\newcommand{\Q}{\mathbb{Q}}
\newcommand{\QQ}{\mathbb{Q}}
\newcommand{\C}{\mathbb{C}}
\newcommand{\Z}{\mathbb{Z}}

\newcommand{\N}{\mathbb{N}}
\newcommand{\R}{\mathbb{R}}
\newcommand{\PP}{\mathbb{P}}

\newcommand{\Pic}{\operatorname{Pic}}

\newcommand{\Br}{\operatorname{Branch}}

\newcommand{\rank}{\operatorname{rank}}

\newcommand{\St}{\operatorname{St}}

\newcommand{\Link}{\operatorname{Link}}

\newcommand{\OO}{\mathcal{O}}

\newcommand{\Diff}{\operatorname{Diff}}

\newcommand{\Addresses}{{% additional braces for segregating \footnotesize
  \bigskip
  \footnotesize
  \noindent \textsc{Department of Mathematics, Imperial College London, 180 Queen’s Gate, London SW72AZ. }\newline
  \noindent\textit{E-mail address}: \texttt{m.mauri15@imperial.ac.uk}
}}

\makeatletter
\newcommand{\xdashrightarrow}[2][]{\ext@arrow 0059\rightarrowfill@@{#1}{#2}}
\newcommand{\xdashleftarrow}[2][]{\ext@arrow 3095\leftarrowfill@@{#1}{#2}}
\newcommand{\xdashleftrightarrow}[2][]{\ext@arrow 3359\leftrightarrowfill@@{#1}{#2}}
\def\rightarrowfill@@{\arrowfill@@\relax\relbar\rightarrow}
\def\leftarrowfill@@{\arrowfill@@\leftarrow\relbar\relax}
\def\leftrightarrowfill@@{\arrowfill@@\leftarrow\relbar\rightarrow}
\def\arrowfill@@#1#2#3#4{%
  $\m@th\thickmuskip0mu\medmuskip\thickmuskip\thinmuskip\thickmuskip
   \relax#4#1
   \xleaders\hbox{$#4#2$}\hfill
   #3$%
}
\makeatother
\renewbibmacro{in:}{%
    \ifentrytype{article}{}{\printtext{\bibstring{in}\intitlepunct}}}
\usepackage[babel]{csquotes}
\usepackage{guit}
\DeclareFieldFormat
  [article,inbook,incollection,inproceedings,patent,thesis,unpublished]
  {title}{{#1\isdot}}
\DeclareFieldFormat{pages}{#1}
\DefineBibliographyStrings{english}{andothers={}}
\usepackage{cancel}

\AtEveryBibitem{%
\ifentrytype{book}{
    \clearfield{url}%
    \clearfield{pages}
    \clearfield{isbn}
    \clearfield{urldate}%
    \clearfield{review}%
    \clearfield{series}%%
}{}
\ifentrytype{article}{
    \clearfield{url}%
    \clearfield{issn}
    \clearfield{doi}
    \clearfield{isbn}
    \clearfield{urldate}%
    \clearfield{review}%
}{}
\ifentrytype{collection}{
    \clearfield{url}%
    \clearfield{urldate}%
    \clearfield{review}%
}{}
\ifentrytype{incollection}{
    \clearfield{url}%
    \clearfield{urldate}%
    \clearfield{review}%
}{}
}

\title{The dual complex of log Calabi--Yau pairs on Mori fibre spaces}

\author{Mirko Mauri}

\begin{document}
\maketitle

\maketitle

\begin{abstract}
In this paper we show that the dual complex of a dlt log Calabi--Yau pair $(Y, \Delta)$ on a Mori fibre space $\pi: Y \to Z$ is a finite quotient of a sphere, provided that either the Picard number of $Y$ or the dimension of $Z$ is $\leq 2$. This is a partial answer to Question 4 in \cite{KollarXu2016}.
\end{abstract}
\tableofcontents
\section{Introduction}
A dual complex is a cell complex, encoding the combinatorial data of how the irreducible components of a simple normal crossing or dlt boundary intersect. These objects have raised the interest of many scholars in different fields. For instance, the homeomorphism type of the dual complex of a minimal dlt modification is an interesting invariant of a singularity; see \cite{deFernexKollarXu2017}. In mirror symmetry, the dual complex of the special fibre of a good minimal dlt degeneration of Calabi--Yau varieties has recently been proved to be the base of a non-archimedean SYZ fibration; see \cite{NicaiseXu2016a} and \cite{NicaiseXuYu2018}. 

In both these examples, a neighbourhood of any cell of the dual complex is a cone over the dual complex of a smaller dimensional dlt pair $(X, \Delta)$, which satisfies the additional property that $K_X+\Delta \sim_\Q 0$, provided that the singularity is log canonical and the degeneration semistable. These pairs are called log Calabi--Yau, in brief logCY. Their dual complexes have been deeply studied in \cite{KollarXu2016}. In that paper, the authors have posed the question whether the dual complex of a logCY pair of dimension $n+1$ is the quotient of a sphere $\mathbb{S}^k$ of dimension $k \leq n$ for some finite group $G \subset \operatorname{O}_{k+1}(\R)$. With the techniques developed, they were able to provide a positive answer in dimension $\leq 4$, and in dimension $= 5$ under the additional hypothesis that $(X, \Delta)$ is simple normal crossing. It is worthy to remark that an affirmative answer to this question would imply, for instance, that the base of a SYZ fibration has the structure of a topological orbifold.

In this paper, we answer positively the question for a special class of dlt logCY pairs $(Y, \Delta)$, endowed with a morphism $\pi\colon Y \to Z$ of relative Picard number one. This hypothesis is inspired by the following observation. If $(X, \Delta_X)$ is a logCY pair, then $K_X$ is not pseudo-effective. % with maximal intersection, \emph{i.e.} the pair admits a 0-dimensional lc centre, then $X$ is rationally connected; see \cite[\S 18]{KollarXu2016}. 
By \cite{BirkarCasciniHaconEtAl2010}, a $K_X$-MMP with scaling $f \colon X \dashrightarrow Y$ terminates with a Mori fibre space $\pi\colon Y \to Z$ and the pair $(Y, \Delta := f_*\Delta_X)$ is still logCY. %However, note that in general the pair $(Y, \Delta)$ is not dlt, even if $(X, \Delta_X)$ is so. Notwithstanding, 
It sounds sensible to the author to check first whether the dual complexes of these special pairs are finite quotient of spheres, under the dlt assumption.

In this paper, we describe the dual complex of dlt logCY pairs $(Y, \Delta)$ on Mori fibre spaces $\pi: Y \to Z$, subject to the condition that either the Picard number of $Y$ or the dimension of $Z$ is smaller or equal to two, \textit{i.e.} $\rho(Y)\leq 2$ or $\dim Z \leq 2$. The main results are collected in the following statement.
\begin{thm}\label{mainthm}
 Let $(Y, \Delta)$ be a dlt pair such that:
 \begin{enumerate}
 \item $Y$ is a $\Q$-factorial projective variety of dimension $n+1$;
 \item (Mori fibre space) $\pi\colon Y \to Z$ is a Mori fibre space of relative dimension $r \coloneqq \dim Y - \dim Z$;
 \item (logCY) $K_Y+\Delta \sim_\Q 0$.
 %\item (maximal intersection) $\dim \mathcal{D}( \Delta)=n$.
 \end{enumerate}
 
  If $\rho(Y)=1$, then the dual complex $\mathcal{D}( \Delta)$ is PL-homeomorphic either to a closed ball $\mathbb{B}^{m}$ of dimension $m \leq n$ or to the sphere $\mathbb{S}^n$.
 
 If $\rho(Y)=2$, then $\mathcal{D}( \Delta)$ is PL-homeomorphic either to a closed ball $\mathbb{B}^{m}$ of dimension $m \leq n$ or to a sphere $\mathbb{S}^{m}$ of dimension $m=r-1, n-r$ or $n$.
 
 If $\dim Z =2$, then $\mathcal{D}( \Delta)$ is PL-homeomorphic to a closed ball $\mathbb{B}^{m}$ of dimension $m \leq n$, to a sphere $\mathbb{S}^m$ of dimension $m=1,n-2, n-1$ or $n$, or to the quotient $\PP^2(\R)*\mathbb{S}^{n-3}$. 
  
  All these cases occur. 
\end{thm}

\begin{proof}[Proof of Theorem \ref{mainthm}] 
%The proof is divided in several steps. First, we consider the case $\rho(Y)=1$ (Theorem \ref{conjsncFanoPic1}). The next case, namely $\rho(Y)=2$, relies on the previous, and it leads to the following alternative: either there exists a vertical stratum (i.e. intersection of components of $\Delta$) of maximal dimension which is not a vertical divisor, or not. We studied these two subcases, separately in \S \ref{Reduction of vertical strata of maximal dimension to vertical divisors} and \S \ref{Vertical strata of maximal dimension are vertical divisors} respectively. Clearly, the case $\rho(Y)=2$ includes the case $\dim Z=1$. It remains to discuss the case of Mori fibre spaces over surfaces, and this is done in Theorem \ref{DualComplexlogCYMorifibrespace}.

It is a combination of the following theorems:
\begin{enumerate}
\item Theorem \ref{conjsncFanoPic1} for $\rho(Y)=1$;
\item Theorem \ref{DualComplexlogCYMorifibrespace} for $\rho(Y)=2$ (which includes in particular the case $\dim Z =1$); 
\item Theorem \ref{dualcomplexsurface} for $\dim Z =2$. 
\end{enumerate}

Finally, all the cases can be realized as suitable hyperplane arrangements in $\PP^{n+1}$, $\PP^{n-r+1}\times \PP^{r}$, or $E \times \PP^1 \times \PP^{n-1}$, where $\PP^k$ is the projective space of dimension $k$ and $E$ is an elliptic curve, with the only exception of the PL-homeomorphism type $\PP^2(\R)*\mathbb{S}^{n-3}$, which is discussed in Example \ref{examplePR2Sn-3}; see also Table \ref{table:example}.
\end{proof}

Theorem \ref{mainthm} can be summarised in the following tables:
\begin{table}[h]
   \begin{tabular}{ c c l}
$\rho(Y)$ & $\dim(Z)$ & PL-homeomorphism type of $\mathcal{D}( \Delta)$ \\ \hline
1& & $\mathbb{B}^m, \mathbb{S}^n$\\
2& & $\mathbb{B}^m, \mathbb{S}^{r-1}, \mathbb{S}^{n-r}, \mathbb{S}^n$\\
& 1 & $\mathbb{B}^m, \mathbb{S}^{0}, \mathbb{S}^{n-1}, \mathbb{S}^n$\\
& 2 & $\mathbb{B}^m, \mathbb{S}^{1}, \mathbb{S}^{n-2}, \mathbb{S}^{n-1}, \mathbb{S}^n, \PP^2(\R)*\mathbb{S}^{n-3}$
    \end{tabular}
    \vspace{0.3 cm}
   \caption{Dual complex of logCY pairs on Mori fibre spaces.}
    \end{table}\\
    
    \begin{table}[h]
    \renewcommand{\arraystretch}{1.3}
       \begin{tabular}{ c c c }
        \multicolumn{1}{ }{}{} \\
        %\hline
        & $\mathcal{D}(\Delta)$ & Example of logCY pair $(Y, \Delta)$ with prescribed $\mathcal{D}(\Delta)$ \\ \hline 
         \multirow{2}{*}{$\rho(Y)=1$} & \multirow{2}{*}{$\mathbb{B}^m$} & $(\PP^{n+1}, \sum^{m+1}_{i=1} \Delta_i + \sum^{2(n+1-m)}_{j=1} \frac{1}{2} \Delta_j)$ \\
         && with $\Delta_{\bullet} \in |\mathcal{O}_{\PP^{n+1}}(1)|$ \vspace{0.3cm}\\
        
        $\rho(Y)=1$ & $\mathbb{S}^n$ & $(\PP^{n+1}, \sum^{n+2}_{i=1} \Delta_i)$ with $\Delta_i \in |\mathcal{O}_{\PP^{n+1}}(1)|$\vspace{0.3cm} \\
        
         \multirow{2}{*}{$\rho(Y)=2$} &  \multirow{2}{*}{ $\mathbb{B}^{a+b+1}$} & 
        $(\PP^{n-r+1}, \sum^{a+1}_{i=1} \Delta_i + \sum^{2(n-r+1-a)}_{j=1} \frac{1}{2} \Delta_j) \times (\PP^{r}, \sum^{b+1}_{i=1} \Delta_i$ 
        \\
        && $ + \sum^{2(r-b)}_{j=1} \frac{1}{2} \Delta_j)$ with $\Delta_{\bullet} \in |\mathcal{O}_{\PP}(1)|$\vspace{0.3cm} \\
        
        \multirow{2}{*}{$\rho(Y)=2$} &  \multirow{2}{*}{ $\mathbb{S}^{r-1}$} & $(\PP^{n-r+1}, \Delta') \times (\PP^{r}, \sum^{r+1}_{i=1} \Delta_i)$ \\
        && with $\Delta_i \in |\mathcal{O}_{\PP^r}(1)|$ and $(\PP^{n-r+1}, \Delta')$ any klt logCY pair \vspace{0.3cm}\\
        
         \multirow{2}{*}{$\rho(Y)=2$} &  \multirow{2}{*}{ $\mathbb{S}^{n-r}$} & $(\PP^{n-r+1}, \sum^{n-r+2}_{i=1} \Delta_i) \times (\PP^{r}, \Delta')$ \\
            & & with $\Delta_i \in |\mathcal{O}_{\PP^{n-r+1}}(1)|$ and $(\PP^{r}, \Delta')$ any klt logCY pair \vspace{0.3cm}\\
        
         $\rho(Y)=2$ & $\mathbb{S}^n$ & $(\PP^{n-r+1}, \sum^{n-r+2}_{i=1} \Delta_i) \times (\PP^{r}, \sum^{r+1}_{j=1} \Delta_j)$ with $\Delta_i \in |\mathcal{O}_{\PP}(1)|$\vspace{0.3cm}\\
         
          \multirow{2}{*}{$\dim Z=1$} & $\mathbb{B}^m, \mathbb{S}^{0},$ & \multirow{2}{*}{as in case $\rho(Y)=2$, when $r=n$} \\
          & $ \mathbb{S}^{n-1}, \mathbb{S}^n$ &  \vspace{0.3cm}\\
          \multirow{2}{*}{$\dim Z=2$} & $\mathbb{B}^m, \mathbb{S}^{1},$ & \multirow{2}{*}{as in case $\rho(Y)=2$, when $r={n-1}$} \\
                & $ \mathbb{S}^{n-2}, \mathbb{S}^{n}$ &  \vspace{0.3cm}\\
                
            \multirow{2}{*}{$\dim Z=2$} &  \multirow{2}{*}{$\mathbb{S}^{n-1}$} & $ E \times (\PP^1, \Delta_0 + \Delta_1) \times (\PP^{n-1}, \sum^{n}_{j=1}\Delta_j)$  \\
            && with $\Delta_{\bullet} \in |\mathcal{O}_{\PP^{n+1}}(1)|$ \vspace{0.3cm}\\
           
           $\dim Z=2$ & $\PP^2(\R)*\mathbb{S}^{n-3}$ & Example \ref{examplePR2Sn-3}. \vspace{0.3cm}
        
        \end{tabular}
        \vspace{0.3 cm}
        \caption{\label{table:example} Examples of logCY pairs $(Y, \Delta)$ which admits a structure of Mori fibre space (in all cases just a projection from a product) and such that the PL-homemorphism type of the dual complex $\mathcal{D}(\Delta)$ is prescribed. The notation $\Delta_i \in |\mathcal{O}_{\PP^{n+1}}(1)|$ indicates that the divisor $\Delta_i$ is a general element of the linear system $|\mathcal{O}_{\PP^{n+1}}(1)|$. Note that all the PL-homemorphism types listed in Table 1 occur.}
        \end{table}        
In particular observe that all these dual complexes are quotients of spheres, compatibly with the prediction \cite[Question 4]{KollarXu2016}.
   
The main ingredients of the proof of Theorem \ref{mainthm} are various connectivity theorems. The first of them is the Hodge index theorem: ample divisors always intersect, provided that they have dimension at least one. This fact allows in \S \ref{Dual complex of logCY pairs with Picard number $1$} to list all the triangulations of $\mathcal{D}( \Delta)$ under the assumption $\rho(Y)=1$. The naive idea for the next step, namely the case of $\rho(Y)=2$, would be to build $\mathcal{D}( \Delta)$ out of the contribution of vertical divisors together with the information provided by horizontal divisors. Indeed, the pushforward of the former determines a logCY pair $(Z, B)$ of Picard number one, while the latter cut out a logCY pair $(F_{\text{gen}}, \Delta_{\text{gen}})$ on a general fibre $F_{\text{gen}}$, which in turns behaves like a logCY pair of Picard number one. The special pairs for which this program works are here called \textit{of combinatorial product type}. The proof of Theorem \ref{mainthm} consists precisely in understanding how far the general pair $(Y, \Delta)$ is from this ideal arrangement. 

As a measure of what can go wrong, observe that there could be strata of $(Y, \Delta)$ which do not dominate $Z$, but which are not contained in any vertical divisor of $\Delta$. This instance is analysed in \S \ref{Reduction of vertical strata of maximal dimension to vertical divisors}. Another issue is represented by horizontal strata which map two-to-one to $Z$: Section \S \ref{Vertical strata of maximal dimension are vertical divisors} accounts for them. These latter strata are also responsible for the occurrence of the homeomorphism type $\PP^2(\R)*\mathbb{S}^{n-3}$, as explained in \S \ref{dimension2case}. We point out that the proof of Theorem \ref{mainthm} highly relies on the connectivity theorems \cite[Proposition 4.37]{Kollar2013a} and \cite[Theorem 4.40]{Kollar2013a} and the canonical bundle formula \cite[Theorem 8.5.1]{Kollar2007}. Observe finally that for a statement which does not involve non-trivial quotients of spheres, our hypothesis on the Picard number is the sharpest possible. Indeed, in Example \ref{examplePR2Sn-3} we construct logCY pairs on Mori fibre spaces of Picard rank three such that $\mathcal{D}( \Delta)\simeq \PP^2(\R)*\mathbb{S}^{n-3}$.

\subsection*{Acknowledgement} I would like to thank Fabio Bernasconi, Paolo Cascini, Tommaso de Fernex, Stefano Filipazzi and Roberto Svaldi for useful discussions. In particular, I would like to express my gratitude to my advisor Paolo Cascini and Fabio Bernasconi for reading an early draft of the paper and correcting some mistakes and imprecisions.  I am also grateful to the anonymous referee for many useful suggestions and corrections.

This work was supported by the Engineering and Physical Sciences Research  Council  [EP/L015234/1],  The  EPSRC  Centre  for  Doctoral  Training  in Geometry  and  Number  Theory  (The  London  School  of  Geometry  and  Number Theory), University College London and Imperial College, London.

\section{Notation: birational dictionary}\label{birationaldictionary}
\subsection{} We work over an algebraically closed field in characteristic zero. A log pair $(Y, \Delta)$ is the datum of a normal variety $Y$ and a $\Q$-divisor $\Delta$ such that $K_Y + \Delta$ is $\Q$-Cartier. If all coefficients of $\Delta$ are in $(0,1]$ (resp. $(-\infty, 1]$), we say that $\Delta$ is a boundary (resp. a sub-boundary). Its support  is the union of the prime divisors with non-zero coefficient in $\Delta$. If $\Delta$ is a sub-boundary, then 
$\Delta=\Delta^{=1} + \Delta^{<1}$,
where $\Delta^{=1}$ and $\Delta^{<1}$ are the sums of the irreducible divisors which appear in $\Delta$ with coefficient equal to one or smaller than one respectively.
 
A $\Q$-divisor is \textbf{$\Q$-Cartier} if one of its multiples is Cartier. A normal variety $Y$ is \textbf{$\Q$-factorial} if any Weil divisor on $Y$ is $\Q$-Cartier.

Let $f\colon X \rightarrow Y$ be a birational morphism. Given a log pair $(Y,\Delta)$, its \textbf{log pull-back} via $f$ is the log pair $(X, \Delta_X)$ determined by the relations
\[K_X + \Delta_X \sim_{\Q} f^*(K_Y+\Delta) \qquad f_*\Delta_X=\Delta.\]
The negative of the coefficient of a prime divisor $E$ in $\Delta_X$, labelled $a(E, Y, \Delta)$, is its \textbf{discrepancy}. 

\subsection{} A log pair $(Y, \Delta)$ is \textbf{log canonical}, abbreviated lc, if $a(E, Y, \Delta)\geq -1$ for any $f\colon X \rightarrow Y$ birational morphism and for any divisor $E \subset X$. An irreducible subvariety $W \subset Y$ is a \textbf{lc centre} if there exists a birational morphism $f\colon X \rightarrow Y$ and a divisor $E \subset X$, called \textbf{lc place}, whose discrepancy $a(E, Y, \Delta)$ equals $-1$ and whose image coincides with $W$.

A log pair $(Y, \Delta)$ is \textbf{log smooth} or \textbf{simple normal crossing}, abbreviated snc, if $Y$ is a smooth variety and the support of $\Delta$ has simple normal crossings. Given any log pair, there is a largest open subset $Y^{\text{snc}}\subset Y$, called simple normal crossing locus, such that $(Y^{\text{snc}}, \Delta|_{Y^{\text{snc}}})$ is snc.

A log canonical pair $(Y, \Delta)$ is dlt, alias \textbf{divisorial log terminal}, if none of the lc centres is contained in $Y \setminus Y^{\text{snc}}$. %The lc centres of $(Y, \Delta)$ are also called \textbf{strata} of $\Delta$.
A log canonical pair $(Y, \Delta)$ with no lc centre is klt, alias \textbf{Kawamata log terminal}. 
A log canonical pair $(Y, \Delta)$ is qdlt, alias \textbf{quotient divisorial log terminal}, if for any lc centre $W$ of codimension $d$ there are $\Q$-Cartier divisors $\Delta_1, \ldots, \Delta_d \in \Delta^{=1}$ containing $W$; see also \cite[Proposition 34]{deFernexKollarXu2017}. 

The lc centres of a (q)dlt pair $(Y, \Delta)$ are the connected components of the intersection of the irreducible divisors in the support of $\Delta^{=1}$; see \cite[Theorem 4.16]{Kollar2013a}. Equivalently, we call the lc centres of the (q)dlt pair \textbf{strata} of $\Delta$.

  If $(Y, \Delta:=\sum_{i}\Delta_i)$ is a dlt pair, then for every lc centre $W$ there exists a unique $\Q$-divisor $\Diff^*_W(\Delta)$ on $W$, called \textbf{different}, with the following property: for $m \in \N$ divisible enough, the Poincar\'{e} residue map $
      \omega^{[m]}_Y(m \Delta)|_W \simeq \omega^{[m]}_W$ on the snc locus extends to the isomorphism $\omega^{[m]}_Y(m \Delta)|_W \simeq \omega^{[m]}_W(\Diff^*_W(\Delta))$ on the locus where $\omega^{[m]}_Y(m \Delta)|_W$ and $\omega^{[m]}_W$ are locally free; see \cite[\S 4.18]{Kollar2013a}. In particular, this yields
      $$(K_Y + \Delta)|_W \sim_{\Q} K_W + \Diff^*_W(\Delta),$$ and by adjunction we have that 
    \begin{equation*}\label{different} \Diff^*_W(\Delta)=\sum_{W \nsubseteq D_i }D_i|_W + \Diff^*_W(\Delta)^{<1}.
    \end{equation*}
  
  \subsection{} A lc pair $(Y, \Delta)$ is logCY, alias \textbf{log Calabi--Yau}, if $K_Y+\Delta \sim_\Q 0$. In particular, let $f \colon X \to Y$ be a \textbf{quasi-\'{e}tale} map, \emph{i.e.} a finite map which is \'{e}tale away from a codimension $\geq 2$ subset. Then, the pair $(X, \Delta_X := f^*\Delta)$ is logCY if and only if the pair $(Y, \Delta)$ is logCY; see also \cite[Proposition 5.20]{KollarMori1998}.

\subsection{} A \textbf{Mori fibre space} for the log pair $(Y, \Delta)$ is an algebraic fibre space $\pi\colon Y \to Z$, i.e. a surjective projective morphism with connected fibres, satisfying the following properties: 
\begin{enumerate}
\item $\dim Z < \dim Y$;
\item the relative Picard number $\rho(Y/Z)$ is one;
\item the $\Q$-divisor $-(K_Y+\Delta)$ is $\pi$-ample.
\end{enumerate}
If no log pair is mentioned, we tacitly assume $\Delta=0$.

\begin{prop}\label{qdltpair}\emph{\cite[Proposition 5.5]{HogadiXu009}, \cite[Proposition 40]{deFernexKollarXu2017}}.
Let $(Y, \Delta)$ be a $\Q$-factorial (q)dlt pair. If $\pi\colon Y \to Z$ is a Mori fibre space for the log pair $(Y, \Delta)$ and the support of $\Delta$ does not dominate $Z$, then the pair $(Z, \pi(\Delta))$ is qdlt. In particular, $Z$ is klt.
\end{prop}

\section{Notation: simplicial complexes}\label{sectionsimplicialcomplex}
 
\subsection{} We establish the notation:
 \begin{itemize}
  \item $\mathbb{S}^n:=\{(x_0, \ldots, x_n)\in \R^{n+1}|\, \sum_i x^2_i=1\}$ is the sphere of dimension $n$;
  \item $\mathbb{B}^n:=\{(x_1, \ldots, x_n)\in \R^{n}|\, \sum_i x^2_i\leq 1\}$ is the (closed) ball of dimension $n$;  
  \item $\PP^n(\R)$ is the real projective space of dimension $n$;
  \item $\sigma^n:=\{(x_0, \ldots, x_n)\in \R^{n+1}|\, \sum_i x_i=1,\, x_i\geq 0\}$ is the (standard) \textbf{simplex} of dimension $n$;
  \item $\partial \sigma^n:=\{(x_0, \ldots, x_n)\in \R^{n+1}|\, \sum_i x_i=1, \prod_i x_i=0,\, x_i\geq 0\}$ is the boundary of $\sigma^n$;
  \item the \textbf{join} $X*Y$ of two topological spaces $X$ and $Y$ is the quotient space of $X \times Y \times [0,1]$  under the identifications $(x,y_1, 0) \sim (x, y_2, 0)$, for all $x \in X$ and $y_1,y_2 \in Y$, and $(x_1,y, 1) \sim (x_2, y, 1)$, for all $x_1,x_2 \in X$ and $y \in Y$;
  \item the \textbf{suspension} $\Sigma X$ of a topological space $X$ is the join $X*\mathbb{S}^0$.
 \end{itemize} 
\begin{defn} \emph{\cite[\S 2.1]{Hatcher2002}}
\begin{itemize}
\item A \textbf{$\Delta$-complex} is the datum of a topological space $\mathcal{D}$ and a triangulation, \emph{i.e.} a collection of characteristic maps $\alpha_j\colon \sigma^d \to \mathcal{D}$, with $d:=d(j)$ depending on $j$, such that:
\begin{enumerate}[label=(\roman*)]
\item the restriction of $\alpha_j$ to the interior $\mathring{\sigma}^d$ of $\sigma^d$ is injective and any point of $\mathcal{D}$ is contained in $\alpha_j(\mathring{\sigma}^d)$, called $d$-dimensional \textbf{cell} or \textbf{face}, for a suitable choice of $j$;
\item for each $(d-1)$-dimensional face of $\partial \sigma^d$, the restriction of the characteristic maps $\alpha_j$ to that face is a characteristic map $\alpha_{j'}\colon \sigma^{d-1} \to \mathcal{D}$;
\item the topology of $\mathcal{D}$ is the coarsest which makes the maps $\alpha_j$ continuous.
\end{enumerate}
\item A \textbf{regular} $\Delta$-complex is a $\Delta$-complex such that in addition the characteristic maps $\alpha_j$ are embeddings.
\item A \textbf{simplicial} complex is a regular $\Delta$-complex such that any two $k$-cells have at most a $(k-1)$-cell in common.
\end{itemize}
\end{defn}
%\begin{figure}
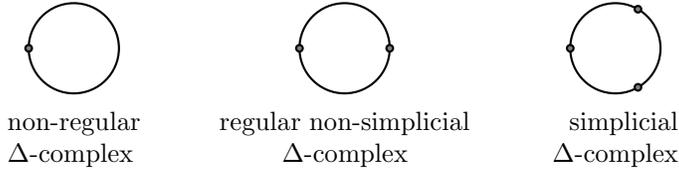
\begin{center}
\tikzstyle{vertex}=[circle, draw, fill=black!50,inner sep=0pt, minimum width=2.5pt]
\begin{tikzpicture}[thick, scale = 1.2]%
\draw (-3,0) circle (0.5 cm);
\draw (-3.5,0) node[vertex]{};
node[vertex]{};\node[align=left, below] at (-3,-.6)%
{non-regular\\ $\Delta$-complex};
\draw (0,0) circle (0.5 cm);
\draw (-0.5,0) node[vertex]{};
\draw (0.5,0) node[vertex]{};
\node[align=center, below] at (0,-.6)%
{regular non-simplicial\\ $\Delta$-complex};
\draw (3,0) circle (0.5 cm);
\draw (2.5,0) node[vertex]{};
\draw (3.25, 0.866025405/2) node[vertex]{};
\draw (3.25, -0.866025405/2)node[vertex]{}; 
\node[align=right, below] at (3,-.6)%
{simplicial\\ $\Delta$-complex};
\end{tikzpicture}
\end{center}
%\caption{}
%\end{figure}
The \textbf{i-th skeleton} of $\mathcal{D}$ is the subcomplex of $\mathcal{D}$ given by the union of all the cells with dimension smaller or equal to $i$. The \textbf{attaching map} of a cell is the restriction $a_j|_{\partial \sigma^{d}}\colon \partial \sigma^{d} \to \mathcal{D}$ of its characteristic map. %Observe also that the join of $\Delta$-complexes $\mathcal{D}_1$ and $\mathcal{D}_2$ has a natural structure of $\Delta$-complex whose cells are join of cells of $\mathcal{D}_1$ and $\mathcal{D}_2$ respectively; see \cite[\S 3.1.7]{Spanier1966}. Indeed, $\sigma^n * \sigma^m \simeq \sigma^{n+m+1}$.

Note that a $\Delta$-complex is prescribed both by the set of its cells and their attaching maps. However, in order to define a regular $\Delta$-complex the datum of the attaching maps is redundant: it is enough to provide the poset of its cells. 

Let $\mathcal{D}$ be a regular $\Delta$-complex. If $v \subset \mathcal{D}$ is a cell of $\mathcal{D}$ we define
\begin{itemize}
\item the \textbf{(open) star} of $v$, denoted $\St(v)$, as the union of the interiors of the cells whose closure intersects $v$; 
\item the \textbf{closed star} of $v$, denoted $\overline{\St(v)}$, is the closure of $\St(v)$;
\item the \textbf{link} of $v$, denoted $\Link(v)$, is the difference $\overline{\St(v)}\setminus \St(v)$ (in the first barycentrical subdivision of $\mathcal{D}$, if $\mathcal{D}$ is not simplicial).
\end{itemize}
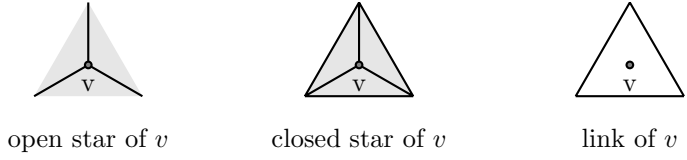
\begin{center}
\tikzstyle{vertex}=[circle, draw, fill=black!50,inner sep=0pt, minimum width=2.5pt]
\begin{tikzpicture}[thick, scale = 1.2]%
\fill[fill=gray!20] (-3+1.2*0.5, -1.2*0.28867513459)--(-3-1.2*0.5, -1.2*0.28867513459)--(-3, 1.2*0.577350268);
\draw{
(-3,0) node{}-- (-3+1.2*0.5, -1.2*0.28867513459)
(-3,0) node{}-- (-3-1.2*0.5, -1.2*0.28867513459)
(-3,0) node{}-- (-3, +1.2*0.577350268)
(-3,0) node[vertex, label=below: v]{}
};

\fill[fill=gray!20] (1.2*0.5, -1.2*0.28867513459)--(-1.2*0.5, -1.2*0.28867513459)--(0, 1.2*0.577350268);
\draw{
(0,0) node{}-- (1.2*0.5, -1.2*0.28867513459)
(0,0) node{}-- (-1.2*0.5, -1.2*0.28867513459)
(0,0) node{}-- (0, 1.2*0.577350268)
(1.2*0.5, -1.2*0.28867513459) node{}--(-1.2*0.5, -1.2*0.28867513459)
(0, 1.2*0.577350268) node{}--(-1.2*0.5, -1.2*0.28867513459)
(0, 1.2*0.577350268) node{}--(+1.2*0.5, -1.2*0.28867513459)
(0,0) node[vertex, label=below: v]{}
};

\draw{
(3+1.2*0.5, -1.2*0.28867513459) node{}--(3-1.2*0.5, -1.2*0.28867513459)
(3, 1.2*0.577350268) node{}--(3-1.2*0.5, -1.2*0.28867513459)
(3, 1.2*0.577350268) node{}--(3+1.2*0.5, -1.2*0.28867513459)
(3,0) node[vertex, label=below: v]{}
};
\node[align=left, below] at (-3,-.6)%
{open star of $v$};
\node[align=center, below] at (0,-.6)%
{closed star of $v$};
\node[align=right, below] at (3,-.6)%
{link of $v$};
\end{tikzpicture}
\end{center}

A map between $\Delta$-complexes is \textbf{piecewise-linear}, abbreviated PL, if the restriction of the map to any face is a linear application onto a face of the target, up to refinements of the triangulations.

\subsection{}\label{definition dual complex} The \textbf{dual complex} of a reduced snc pair $(Y, \Delta)$, denoted $\mathcal{D}( \Delta)$, is the regular $\Delta$-complex whose vertices are in correspondence with the irreducible components of $\Delta$ and whose $d$-faces $v_W$ correspond to a lc centre $W$ of codimension $d+1$.

The dual complex of the dlt pair $(Y, \Delta)$ is the dual complex of the snc pair $(Y^{\text{snc}}, \Delta|^{=1}_{Y^{\text{snc}}})$; see \cite[\S 2]{deFernexKollarXu2017}. Hence, $\mathcal{D}(\Delta)=\mathcal{D}(\Delta^{=1})$ by definition. We define the dual complex of a log canonical pair $(Y, \Delta)$ as the PL-homeomorphism type of the dual complex of a dlt modification of $(Y, \Delta)$; see \cite[Theorem 1.34]{Kollar2013a}.
If $(Y, \Delta)$ is qdlt, the same construction performed for snc or dlt pairs gives a regular $\Delta$-complex which coincides with the one provided by a dlt modification; see \cite[Corollary 38]{deFernexKollarXu2017}.

Let $v_W$ be a $d$-face of $\mathcal{D}( \Delta)$ associated to the lc centre $W$. Up to barycentrical subdivisions, $\Link(v_W)$ can be identified with the dual complex of the trace of $\Delta$ on $W$, \emph{i.e.} $\sum_{W \nsubseteq D_i }D_i|_W$. In symbols,
\begin{equation}\label{linkdifferent}
  \Link(v_W)\simeq_{\text{PL}} \mathcal{D}( \Diff^*_W(\Delta))\simeq \mathcal{D}\big(\sum_{W \nsubseteq D_i }D_i|_Y\big).
\end{equation}
\begin{defn}
The dlt pair $(Y, \Delta)$ of dimension $n+1$ has \textbf{maximal intersection} if it admits a 0-dimensional lc centre $W$. Equivalently, the corresponding face $v_W$ has dimension $n$ and we say that $\mathcal{D}( \Delta)$ has \textbf{maximal dimension}.
\end{defn}

If the projective dlt pair $(Y, \Delta)$ has maximal intersection, then $Y$ is rationally connected and $H^i(Y, \mathcal{O}_Y)=0$ for any $i>0$; see \cite[Proposition 19]{KollarXu2016}.
   
\section{Dual complex of logCY pairs with Picard number 1}\label{Dual complex of logCY pairs with Picard number $1$}
 
 In this section we describe the explicit structure of $\Delta$-complex of the dual complex of a qdlt logCY pair with Picard number one.
 
 \begin{thm}\label{conjsncFanoPic1}
 Let $(Y, \Delta)$ be a qdlt pair such that:
 \begin{enumerate}[label=(\roman*)]
 \item $Y$ is a %$\Q$-factorial 
 projective variety of dimension $n+1$; 
 \item each
 irreducible component of $\Delta^{=1}$ is ample (e.g.  $\rho(Y)=1$);
 \item (logCY) $K_Y+\Delta \sim_\Q 0$.
 \end{enumerate}
 
 Then $\mathcal{D}( \Delta)$ is PL-homeomorphic either to a ball $\mathbb{B}^{m}$ of dimension $m \leq n$ or to the sphere  $\mathbb{S}^n$.    
 
 More precisely, $\mathcal{D}( \Delta)$ is isomorphic to one of the following regular $\Delta$-complexes:
 \begin{enumerate}
 \item\label{simplex} (standard simplex) standard simplex $\sigma^{m}$ of dimension $m\leq n$;
 \item\label{simplicialsphere} (simplicial $n$-sphere) boundary $\partial \sigma^{n+1}$ of the standard simplex $\sigma^{n+1}$;
 \item\label{regularsphere} (non-simplicial $n$-sphere) union of two standard simplexes $\sigma^{n}$, glued along the boundary.
 \end{enumerate} 
 \end{thm}
 \begin{proof}
Let $\Delta^{=1} = \sum_{i=1}^{m+1}\Delta_i$. Any  stratum $W$ of $\Delta^{=1}$  
is a connected (irreducible) component of $\Delta_{i_1}\cap \ldots \cap \Delta_{i_s}$ for certain $I=\{i_1, \ldots, i_s\}\subseteq \{1,\ldots, m+1\}$. The restriction to $W$ of $\Delta_i$ for $i \notin I$, denoted $\Delta_i|_W$, is an ample divisor in $W$. By the Hodge index theorem,  $\Delta_i|_W$ is non-empty and connected, as long as $\dim W\geq 2$. Indeed, suppose on the contrary that an ample divisor is not connected; by slicing with general hyperplane sections, we can reduce to the case of a (possibly singular) surface and derive a contradiction applying the Hodge Index theorem to the resolution of the normalization of the surface. 

\textit{Case} $(m<n)$. Suppose that $\Delta^{=1}$ does not contain any $0$-dimensional strata, or equivalently $m<n$. Since all $\Delta_i|_W$ are ample and intersect in a unique connected component (we follow the same notation of the previous paragraph), any collection of $s+1$ 0-cells of $\mathcal{D}( \Delta)$ is the set of vertices of a unique $s$-simplex. Hence, $\mathcal{D}( \Delta)$ is a $m$-simplex as in (\ref{simplex}).

\textit{Case} $(m\geq n)$. Denote by $D_i$ and $\sigma^{m}_i$ the $i$-th skeleton of $\mathcal{D}( \Delta)$ and $\sigma^{m}$ respectively. Arguing as in the previous case, we note that there are compatible isomorphism $ \sigma^{m}_{n-1} \simeq D_{n-1}$ and inclusion $ \sigma^{m}_{n} \hookrightarrow D_{n}$ of $\Delta$-complexes. We
identify $\sigma^{m}_{n}$ with its image in $D_n$ and regard $\sigma^m_i \subseteq D_i$ for all $i \leq n$.
%we note that:
%\[ D_{n-1}=\sigma^{m}_{n-1} \qquad \qquad D_{n}\supseteq \sigma^{m}_{n}.\]

 Then the long exact sequences in homology of the pairs $(D_{i}, \sigma^{m}_{i})$ with integral coefficient give: 
\begin{enumerate}
\item $H_i(\mathcal{D}( \Delta), \Z)=H_i(D_{n-1})=H_i(\sigma^{m}_{n-1})=H_i(\sigma^{m})=0$ for $0<i<n-1$;
\item $H_{n-1}(\mathcal{D}( \Delta), \Z)=H_{n-1}(D_n)=0$, since the latter group fits into the following exact sequence
\[0=H_{n-1}(\sigma^{m}_n)\to H_{n-1}(D_{n})\to H_{n-1}(D_n, \sigma^{m}_n)=0.\]
The first identity $H_{n-1}(\sigma^{m}_n)=0$ follows from Remark \ref{homologyskeleton}. For the latter vanishing, note that the homology of the pair $(D_n, \sigma^{m}_n)$ is isomorphic to the reduced homology of $D_n/ \sigma^{m}_n$, i.e. the bouquet of $n$-dimensional spheres obtained by identifying $\sigma^{m}_n \subseteq D_n$ (and in particular $D_{n-1}$) to a point: this implies that $H_{i}(D_n, \sigma^{m}_n)=0$ for all $i \neq n$; see \cite[Proposition 2.22]{Hatcher2002}.  
\item \label{item:les} the long exact sequence of the pair $(\mathcal{D}( \Delta), \sigma^{m}_n)$ gives
\[0\to H_{n}(\sigma^{m}_n)\to H_{n}(\mathcal{D}( \Delta))\to H_{n}(\mathcal{D}( \Delta), \sigma^{m}_n)\to 0.\]
%\item the following diagram is exact
%
%\[\xymatrix{
%&&&H_{n-1}(\sigma^{m}_{n})=0&&\\
%0 \ar[r]& H_n(D_n)\ar[r]& H_n(D_n, D_{n-1})\ar[r]& H_{n-1}(D_{n-1})\ar[r]\ar[u]& H_{n-1}(D_n)\ar[r]& 0.\\
%&&&H_{n}(\sigma^{m}_{n},\sigma^{m}_{n-1} )\ar[u]&&\\
%&&&H_{n}(\sigma^{m}_{n})\ar[u]&&\\
%&&&H_{n}(\sigma^{m}_{n-1})=0\ar[u]&&\\
%}
%\]
%%& (Y^n, \Delta^n:=\underbrace{pr_1^*\Delta+ \ldots + pr_n^*\Delta}_{n-\text{times}}) \ar@{->}[d]_{\pi}\\
%(Y^{[n]}, \Delta^{[n]}:= \epsilon^*\Delta^{(n)})\ar@{->}[r]^- \epsilon & (Y^{(n)}, \Delta^{(n)}:=\pi_* \Delta^n)}.

%
%\[0\to H_n(D_n)\to H_n(D_n, D_{n-1})\to H_{n-1}(D_{n-1})\to H_{n-1}(D_n)\to 0.\]
\end{enumerate}  

If $H_n(\mathcal{D}( \Delta))=0$, then $\mathcal{D}( \Delta)= \sigma^{n}$. Indeed, we have
\[
\rank \widetilde{H}_n(\sigma^{m}_n)={{m}\choose{n+1}}=0 \quad \Rightarrow \quad  m \leq n,
\]
which in particular implies $m=n$ and $\sigma^{m}_n=\sigma^{n}$. Then the short exact sequence in (\ref{item:les}) yields
\[
H_n(\mathcal{D}( \Delta))=0 \quad  \Rightarrow \quad H_n(\mathcal{D}( \Delta), \sigma^{n})=0.
\] 
%Together with $(2)$, the assumption $H_n(\mathcal{D}( \Delta))=0$ implies that $H_*(\mathcal{D}( \Delta), \sigma^{n})=0$. 
Finally, we conclude that $\mathcal{D}( \Delta)= \sigma^{n}$:
%\[
%H_*(\mathcal{D}( \Delta), \sigma^{n})=0 \quad  \Rightarrow \quad \mathcal{D}( \Delta)= \sigma^{n},
%\]  
the quantity $\rank H_n(\mathcal{D}( \Delta), \sigma^{m}_n)$ (which equals $\rank \widetilde{H}_n(\mathcal{D}( \Delta)/\sigma^{m}_n)$) measures the number of $n$-cells in $\mathcal{D}( \Delta)$ not contained in our distinguished copy of $\sigma^{m}_n \subseteq D_n = \mathcal{D}( \Delta)$, and in this case there is none of these $n$-cells.
%\begin{align*}
%\rank \widetilde{H}_n(\sigma^{m}_n)={{m}\choose{n+1}}=0 \quad &\Rightarrow \quad  m \leq n;\\
%H_*(\mathcal{D}( \Delta), \sigma^{m})=0 \quad & \Rightarrow \quad \mathcal{D}( \Delta)= \sigma^{n}.
%\end{align*} 
%$\tilde{H}_*(\mathcal{D}( \Delta), \Z)=0$ and $H_*(\mathcal{D}( \Delta), \sigma^{m})=0$, hence $D_n$ is contractible of pure dimension $<n$ (cf. Proof of Theorem \ref{fanocompleteintersectiondualcomplex}, Corollary \ref{puredimensiondualcomplex} and \cite[Claim 32.3]{KollarYu2016}). 

If $H_n(\mathcal{D}( \Delta))\neq 0$, then $\mathcal{D}( \Delta)$ is PL-homemorphic to a sphere. To this end, we first show the following Claim \ref{Claimtopman}. Without loss of generality, suppose that $n>0$ and $\mathcal{D}( \Delta)$ is connected; otherwise, it is either a point, or $\mathcal{D}(\Delta)$ is not connected, in which case $\Delta$ can only be the union of two points on $\PP^1$ for the Hodge Index theorem and the Calabi--Yau assumption; see also the more general argument \cite[Proposition 4.37]{Kollar2013a}.

\begin{claim}\label{Claimtopman}
In the hypothesis of Theorem \ref{conjsncFanoPic1}, and if $H_n(\mathcal{D}( \Delta))\neq 0$, then $\mathcal{D}( \Delta)$ is a closed orientable topological manifold of dimension $n$. 
\end{claim} 
\begin{proof} We argue by induction on dimensions. The base case, namely $n\leq 3$, is assured by dimensional argument; see for instance \cite[\S 33]{KollarXu2016}. Suppose now that the statement of Theorem \ref{conjsncFanoPic1} holds for varieties of dimension $\leq n$. We prove it for $\dim Y=n+1$. 

Note first that the irreducible components of the different $\Diff_{\Delta_i}(\Delta)^{=1}$ are ample divisors in $\Delta_i$, which means that they satisfy the induction hypothesis. 
%of dimension $>3$ is a variety with Picard number 1 by the Lefschetz theorem for Picard groups \cite[ex. 3.1.25]{Lazarsfeld2004}.]
  The link of any cell in $\mathcal{D}( \Delta)$ (eventually after a barycentrical subdivision) is homeomorphic to the dual complex $\mathcal{D}( \Diff^*_{W}(\Delta))$ for some stratum $W$ of codimension $d$; see \ref{definition dual complex} (\ref{linkdifferent}). By induction, this link is PL-homeomorphic to a $(d-1)$-dimensional sphere or a $(d-1)$-dimensional ball. Observe that lower-dimensional balls cannot occur: the condition $H_n(\mathcal{D}( \Delta), \QQ)\neq 0$ implies that $\mathcal{D}(\Delta)$ has maximal dimension, and so $\mathcal{D}( \Diff^*_{W}(\Delta))$ does, due to \cite[Theorem 4.40]{Kollar2013a}. As a result, $\mathcal{D}( \Delta)$ is a (connected) topological manifold eventually with boundary $\partial \mathcal{D}( \Delta)$.
  
  Now, if $H_n(\mathcal{D}( \Delta))\neq 0$ (which implies also $H_n(\mathcal{D}( \Delta), \partial \mathcal{D}( \Delta))\neq 0$), then $\mathcal{D}( \Delta)$ is $\Z$-orientable. By Lefschetz's duality we have \(0 \neq H_n(\mathcal{D}( \Delta))\simeq H^0(\mathcal{D}( \Delta), \partial \mathcal{D}( \Delta))\), which implies $\partial \mathcal{D}( \Delta)=\emptyset$, as required.  
  \end{proof}
  As a corollary, if $H_n(\mathcal{D}( \Delta)) \neq 0$, then $\rank H_n(\mathcal{D}( \Delta))=1$. Further, the long exact sequence of the pair $(\mathcal{D}( \Delta), \sigma^{m}_n)$ gives
\[0\to H_{n}(\sigma^{m}_n)\to H_{n}(\mathcal{D}( \Delta))\simeq \Z\to H_{n}(\mathcal{D}( \Delta), \sigma^{m}_n)\to 0.\]

We are left with two options: either $H_{n}(\mathcal{D}( \Delta), \sigma^{m}_n)=0$ or $\neq 0$. In the former case, $m=n+1$ by Remark \ref{homologyskeleton}, and $\mathcal{D}( \Delta)= \sigma^{n+1}_n= \partial \sigma^{n+1} \simeq \mathbb{S}^n$. Otherwise, $H_{n}(\mathcal{D}( \Delta), \sigma^{m}_n)=\Z$ and $H_{n}(\sigma^{m}_n)=0$, which implies that $n=m$. In other words, $\mathcal{D}( \Delta)$ is obtained by attaching an additional $n$-cell to the standard simplex $\sigma^{n}$, so it is a non-simplicial sphere of dimension $n$.

\end{proof}
%\begin{rmk}\label{weakerhypconjpic1}
%In the proof of Theorem \ref{conjsncFanoPic1}, we could have replaced the hypothesis $\rho(Y)=1$ with the weaker assumption that each irreducible
%component of $\Delta^{=1}$ is an ample divisor. 
%\end{rmk}
\begin{rmk}
In order to prove the first statement of Theorem \ref{conjsncFanoPic1}, after showing the vanishing of the torsion of the homology of $\mathcal{D}( \Delta)$, we could have concluded by invoking the (generalized) Poincar\'{e} conjecture.
\end{rmk}
\begin{rmk}\label{complexity} The proof of Theorem \ref{conjsncFanoPic1} shows that a qdlt anti-canonical divisor on a variety $Y$ with $\rho(Y)=1$ has at most $\dim Y+1$ irreducible components. This is also a consequence of the non-negativity of the complexity of the log pair $(Y, \Delta)$; see \cite[Corollary 1.3]{BrownMcKernanSvaldiEtAl2016}.
\end{rmk}
\begin{rmk}
We bring to the attention of the reader a precedent result by Danilov, \cite[Proposition 3]{Danilov1975}, which states that if $(Y, \Delta)$ is a snc pair of dimension $n+1$ such that at least one of the irreducible components of $\Delta$ is ample, then $\mathcal{D}( \Delta)$ has the homotopy type of a bouquet of $n$-dimensional spheres. The virtue of Theorem \ref{conjsncFanoPic1} is that it provides a complete description of the $\Delta$-complex structure of $\mathcal{D}( \Delta)$, which we will exploit in an essential way in the following.
\end{rmk}

\begin{rmk}\label{boundary<1Pic1}
In the last case $H_n(\mathcal{D}( \Delta))\neq 0$ in the proof, one could avoid the induction argument by computing directly $H_{n}(\mathcal{D}( \Delta))\simeq \Z$; see \cite[claim 32.3]{KollarXu2016}.  The claim 32.3 in \cite{KollarXu2016} also implies that if the dlt logCY pair $(Y, \Delta)$ has maximal intersection, then the following statements are equivalent:
\begin{enumerate}
\item $h_n(\mathcal{D}( \Delta))=1$; 
\item $\Delta^{=1}=\Delta$ and $K_X + \Delta \sim 0$.
\end{enumerate}  Indeed, if $\Delta^{<1}\neq 0$, then $K_Y+\Delta^{=1}\sim_{\Q} -\Delta^{<1}$ and we have
\[h_{n}(\mathcal{D}( \Delta), \C)=h^n(Y, \OO_{\Delta^{=1}})=h^{n+1}(Y, \OO(-\Delta^{=1}))=h^{0}(Y, K_Y+\Delta^{=1})=0,\]
which is a contradiction. Conversely, if $K_X + \Delta^{=1} \sim 0$, the previous sequence of equalities yields $h_n(\mathcal{D}( \Delta))= 1$.
%Note that this implication does not rely on the Picard number assumption. Conversely, if $ \Delta^{<1}=0$, then there exists $r \in \Q$ such that $0 \sim_\Q K_X+ \Delta^{=1}\sim rH$, where $H$ is the ample generator of $\Pic(Y)$. Hence, $r=0$. The previous sequence of equality implies $h_n(\mathcal{D}( \Delta))= 1$.

As a corollary, if $\Delta^{<1} \neq 0$, $\mathcal{D}(\Delta)$ cannot be a $n$-dimensional sphere (since then $h_n(\mathcal{D}( \Delta))\neq 1$), and by Theorem \ref{conjsncFanoPic1} it is a standard simplex; see also \cite[\S 22]{KollarXu2016}.
\end{rmk}
\begin{rmk}\label{homologyskeleton}
The reduced singular homology of the $n$-th skeleton of a standard simplex $\sigma$ of dimension $m$ is concentrated in degree $n$ and
\[
\widetilde{H}_n(\sigma^m_n, \Z)=\Z^{m\choose{n+1}}.
\] 
Indeed, the long exact sequence of the pair $(\sigma^m_{n+1}, \sigma^m_{n})$ gives
\[0=H_{i+1}(\sigma^m_{n+1}, \sigma^m_{n}) \to H_i(\sigma^m_{n})\to H_i(\sigma^m_{n+1}) \quad \text{ for }i\neq n.\]
Recursively, we note that $H_i(\sigma^m_n)\to H_i(\sigma^m)$ is an injective map for $i < n$, but since $H_i(\sigma^m)=0$ for $i\neq 0$, we conclude that $H_i(\sigma^m_n)=0$ for $0<i<n$ and that the reduced homology of $\sigma^m_n$ is concentrated in degree $n$.

The long exact sequence of the pair $(\sigma^m_n, \sigma^m_{n-1})$ gives also
\[0=\widetilde{H}_{n}(\sigma^m_{n-1})\to \widetilde{H}_n(\sigma^m_{n})\to \widetilde{H}_n(\sigma^m_{n}, \sigma^m_{n-1})\to \widetilde{H}_{n-1}(\sigma^m_{n-1})\to \widetilde{H}_{n-1}(\sigma^m_{n})=0\]
Since $\rank H_i(\sigma^m_{i}, \sigma^m_{i-1})={m+1\choose{i+1}}$, we obtain the following formula by recursion
\[\rank \widetilde{H}_n(\sigma^m_n)=(-1)^{n+1}\sum_{i=0}^{n+1}(-1)^i{{m+1}\choose{i}}={m \choose n+1},\]
which can be easily checked using the identity
\[{m \choose n+1}={m-1 \choose n+1}+{m-1 \choose n}.\]
\end{rmk}
%\begin{cor}
%Let $Y$ be a $\Q$-factorial variety of dimension $n+1$ with $\rho(Y)=1$. 
%Let $(X, \Delta)$ be a dlt Calabi-Yau pair. 
%
%Then $\DMR(X, \Delta):=D(X, \Delta^{=1})$ is either contractible of pure dimension $<n$ or a sphere of dimension $n$.
%\end{cor}
%\begin{proof}
%The proof of Theorem \ref{conjsncFanoPic1} continues to work except for small adjustments. Indeed, in the previous proof we have just exploited the fact that each irreducible component of $\Delta^{=1}$ is cut out by an ample divisor and its connectedness in dimension $\geq 2$. This still holds in the singular case. Indeed, suppose on the contrary that an ample divisor is not connected, by slicing with general hyperplane sections, we can reduce to the case of singular surfaces and derive a contradiction applying the Hodge Index theorem to the resolution of the normalization of the surface. 
%\end{proof}
 
%\section{Dual complex of logCY pairs of Fano varieties of Picard number 1}\label{Dual complex of logCY pair with a morphism of relative Picard number 1}

\section{Generalities on the dual complex of logCY pairs on Mori fibre spaces }\label{generalitiesdualcomplexpicard2} Here and in the following, $(Y, \Delta)$ is a $\Q$-factorial dlt logCY pair of dimension $n+1$ with $\Delta^{=1} = \sum_{i=1}^{m+1}\Delta_i$. Suppose that there exists a morphism $\pi\colon Y \to Z$ of relative dimension $r$ and relative Picard number one. The goal is to identify the PL-homeomorphism type of the dual complex $\mathcal{D}( \Delta)$. In this section we collect some general facts about logCY pairs on Mori fibre spaces.

 \begin{defn}\label{defnhorizontal}
 %Let $\pi\colon Y \to Z$ be a morphism. 
 A \textbf{horizontal lc centre} (or a horizontal stratum) of the log pair $(Y, \Delta)$ is a lc centre of $(Y, \Delta)$ which dominates $Z$. A lc centre is said \textbf{vertical} if it is not horizontal. 
 \end{defn}
 Mind that in the sequel horizontal and vertical divisors will always refer to components of $\Delta^{=1}$.
 \begin{defn}
 %In the notation of Definition \ref{defnhorizontal}, 
 A vertical lc centre $W$ is \textbf{maximal} if it is not contained in any other vertical lc centre.
 \end{defn}
 \begin{defn}
 The {dual complexes of the horizontal or vertical divisors} in the support of $\Delta^{=1}$ are denoted $\mathcal{D}^{\text{hor}}$ or $\mathcal{D}^{\text{vert}}$ respectively, and they are regular subcomplexes of $\mathcal{D}( \Delta)$.
 \end{defn}
 Let $(F_{\text{gen}}, \Delta|_{F_{\text{gen}}})$ be the restriction of the logCY pair $(Y, \Delta)$ to a general fibre of $\pi$, denoted $F_{\text{gen}}$. Any stratum of $\Delta|_{F_{\text{gen}}}$ is an irreducible component of the intersection of
 a unique horizontal stratum of $\Delta$ with $F_{\text{gen}}$. Equivalently, the restriction of horizontal strata to $F_{\text{gen}}$ induces a PL-map
 \begin{equation}\label{definition map r}
 \text{hor}:\mathcal{D}( \Delta|_{F_{\text{gen}}} ) \to \mathcal{D}^{\text{hor}}.
 \end{equation}
 Properties of the map $\text{hor}$ are discussed in the following; see in particular next item \ref{numbhorzdivisors}.
 \begin{defn}
 A log pair $(Y, \Delta)$, equivalently $\mathcal{D}( \Delta)$, has \textbf{combinatorial product type} if 
  \begin{enumerate}
  \item any intersection of horizontal divisors is horizontal, \emph{i.e.}
    \[\mathcal{D}^{\text{hor}}=\mathcal{D}( \Delta|_{F_{\text{gen}}} ).\]
    \item any horizontal stratum intersect every vertical strata in a unique connected component, \emph{i.e.}
  \[\mathcal{D}( \Delta)= \mathcal{D}^{\text{vert}}*\mathcal{D}^{\text{hor}}.\]  
  
  \end{enumerate} 
 \end{defn}
 
 Since $\rho(Y/Z)=1$, the following properties hold.
 \begin{enumerate}[label=(\roman*)]
 \item \textit{The image of a vertical divisor is a divisor, and conversely the preimage of an irreducible divisor on $Z$ via $\pi$ is an irreducible divisor}. This is an application of Zariski's lemma (cf. \cite[Theorem 4.14]{Shafarevich2013}) over a codimension one point; see \cite[Lemma 5.1 and 5.2]{HogadiXu009}.
 \item \textit{Horizontal divisors restrict to ample divisors on any fibre}.
 \item \label{numbhorzdivisors} \textit{The intersection of irreducible horizontal divisors is always non-empty and horizontal in codimension $\leq r$, and connected in codimension $\leq r-1$}, since the restriction of horizontal divisors to a stratum of $\Delta|_{F_{\text{gen}}}$ is ample, and by the Hodge index theorem it is non-empty and connected in codimension $\leq r-1$. In particular, the map $\text{hor}$ defined in (\ref{definition map r}) is an isomorphism on the $(r-1)$-th skeleton and it is injective if any horizontal stratum of codimension $r$ intersects vertical strata in a unique connected component.
 \item \label{dualcomplexgeneral fibre}  The $\Delta$-complex structure of $\mathcal{D}( \Delta|_{F_{\text{gen}}})$ is one of those listed in Theorem \ref{conjsncFanoPic1}. In particular, \textit{there are at most $r+1$ horizontal divisors}; see alternatively Remark \ref{complexity}.
 \item \label{verticallogcentre} \textit{A maximal vertical lc centre is either a stratum of codimension $r+1$ or a vertical divisor.} Suppose that $W$ is a maximal vertical lc centre but not a vertical divisor. Then, $W$ is an irreducible component of the intersection of horizontal divisors $\Delta_{i_0}\cap \ldots \cap \Delta_{i_r}$ for certain  $I=\{i_0, \ldots, i_r\}\subseteq \{1,\ldots, m+1\}$, with $|I|=r+1$ by \ref{numbhorzdivisors} and \ref{dualcomplexgeneral fibre}. 
% 
% We are left to check that the former case $|I|=r$ does not occur. Indeed, if 
% \[\dim(\Delta_{i_1}\cap \ldots \cap \Delta_{i_r}|_{F_{\text{gen}}})\leq 0,\]
% then by the connectivity lemma $\Delta_{i_1}\cap \ldots \cap \Delta_{i_r}|_{F_{\text{gen}}}$ is either a single point, in which case $\Delta_{i_1}\cap \ldots \cap \Delta_{i_r}$ is a horizontal stratum and $W$ is not given
%  so that $\cap_{j\in J}\Delta_j$ is empty or not irreducible (equivalently connected) on a general fibre. By $(3)$, this implies that 
%  \[\dim(\Delta_{i_1}\cap \ldots \cap \Delta_{i_r}|_{F_{\text{gen}}})\leq 0.\]
%  Then, either $Z$ consists of a point or the previous intersection is empty, \emph{i.e.} the divisors $D_j=\pi^{-1}(B_i)$ are vertical divisors for certain divisors $B_j \in Z$. By maximality, $W=\pi^{-1}(B_{i_1}\cap \ldots \cap B_{i_r})$.
 \end{enumerate}
 
 \begin{prop}\label{canonicalbundleformula} \emph{\cite[Theorem 8.5.1]{Kollar2007}} \emph{(Canonical bundle formula)}\label{Canonical bundle formula} In the previous hypothesis as at the beginning of the section, we can write
 \[0 \sim_{\Q} K_Y+\Delta \sim_{\Q}\pi^*(K_Z + B + J),\] 
 where 
 \begin{enumerate}
 \item (moduli b-divisor) $J$ is a pseudo-effective $\Q$-linear equivalence class;
 \item \label{coefficientB} (boundary b-divisor) $B$ is a $\Q$-divisor with coefficient along the prime divisor $D$ given by
 \[\operatorname{coeff}_B(D)=\sup_E\left\lbrace 1-\frac{1+a(Y, \Delta, E)}{\operatorname{mult}_E(\pi^*D)}\right\rbrace, \]
 where the supremum is taken over all the divisors $E$ over $Y$ which dominate $D$. In particular, $D$ is dominated by a vertical lc centre if and only if $\operatorname{coeff}_B(D)=1$.
 \end{enumerate}
 \end{prop}
 \begin{rmk} \label{rmk:logCYpaironthebottom}
 If $\rho(Z)=1$ and $B \neq 0$, then there exists a logCY pair $(Z, B+J^\prime)$ such that $B+J^\prime$ is a boundary $\Q$-linearly equivalent to $B+J$ and the $\Q$-divisor $J^\prime$ has coefficient smaller than one. Indeed, under these hypotheses, $Z$ is Fano and it has klt singularities; see for instance \cite[Proposition 5.5]{HogadiXu009}. In particular, by \cite[Corollary 1.13]{HaconMckernan2007} it is rationally chain connected  with rational singularities, thus $H^1(Z, \mathcal{O}_Z)=0$, and there exists $r \in \Q$ such that $J \sim_\Q rH$, where $H$ is the ample generator of $\Pic(Z)$.
  
  If $B=0$ and we are interested in the PL-homeomorphism type of $\mathcal{D}(\Delta)$, we can avoid to induce any logCY structure on $Z$. Indeed, in this case  $\mathcal{D}(\Delta)=\mathcal{D}^{\text{hor}}$, as the condition $B=0$ implies that there are no vertical divisors, and the restriction map $\text{hor}:\mathcal{D}( \Delta_{F_{\text{gen}}}) \twoheadrightarrow \mathcal{D}^{\text{hor}}$ is surjective, since any face in the complement of the image corresponds to vertical intersection of horizontal divisors, which would contribute to $B \neq 0$. Hence, either the map $\text{hor}$ is injective and $\mathcal{D}(\Delta)=\mathcal{D}( \Delta_{F_{\text{gen}}})$ can be computed by means of Theorem \ref{conjsncFanoPic1}, or $\text{hor}$ is not and $\mathcal{D}(\Delta)$ can be described as in case \(B^{=1}=0\) in \S \ref{boundary B empty}.
 \end{rmk}

\section{Dual complex of logCY pairs on Mori fibre spaces with Picard number 2} % $\rho(Y)=2$}
In this section we achieve the goal of identifying the PL-homeomorphism type of $\mathcal{D}( \Delta)$ under the hypotheses at the beginning of \S \ref{generalitiesdualcomplexpicard2}, and the additional condition that $\rank \Pic(Y)=2$. 
\begin{thm}[$\rank \Pic(Y)=2$]\label{DualComplexlogCYMorifibrespace}
 Let $(Y, \Delta)$ be a dlt pair such that:
 \begin{enumerate}
 \item (Picard number two) $Y$ is a $\Q$-factorial projective variety of dimension $n+1$ with $\rho(Y)=2$;
 \item (Mori fibre space) $\pi\colon Y \to Z$ is a Mori fibre space of relative dimension $r$;
 \item (logCY) $K_Y+\Delta \sim_\Q 0$.
 %\item (maximal intersection) $\dim \mathcal{D}( \Delta)=n$.
 \end{enumerate}
 
 Then, $\mathcal{D}( \Delta)$ is PL-homeomorphic either to a ball $\mathbb{B}^{m}$ of dimension $m \leq n$ or to a sphere $\mathbb{S}^{m}$ of dimension $m=r-1, n-r$ or $n$.   
 \end{thm}
 \begin{proof}
 We have the following dichotomy: either any maximal vertical strata is a vertical divisor or not. We study these cases separately in \S \ref{Vertical strata of maximal dimension are vertical divisors} and \S \ref{Reduction of vertical strata of maximal dimension to vertical divisors} respectively. The case with no vertical strata is included in \S \ref{Vertical strata of maximal dimension are vertical divisors}.
 \end{proof}

\subsection{There exists a maximal vertical stratum of codimension $>1$} \label{Reduction of vertical strata of maximal dimension to vertical divisors}  Let $(Y, \Delta)$ be a dlt pair as in Theorem \ref{DualComplexlogCYMorifibrespace}. Denote by $W$ a maximal vertical lc centre which is not a vertical divisor.  %By \ref{generalitiesdualcomplexpicard2}.\ref{verticallogcentre} and Theorem \ref{conjsncFanoPic1}, it is unique. 
The following facts hold.
\begin{enumerate}[label = (\roman*)]
\item\label{defW} \textit{$W$ is an irreducible component of the intersection of all the horizontal divisors in $\Delta$} due to the bound on the number of horizontal divisors in \ref{generalitiesdualcomplexpicard2}.\ref{dualcomplexgeneral fibre} and the description of $W$ in \ref{generalitiesdualcomplexpicard2}.\ref{verticallogcentre}.  In particular, $\mathcal{D}( \Delta|_{F_\text{gen}})\simeq \partial \sigma^{r}$ by Theorem \ref{conjsncFanoPic1}, since the restriction $\Delta|_{F_\text{gen}}$ of $\Delta$ to a general fibre $F_\text{gen}$ of the morphism $\pi$ consists of $r+1$ irreducible components.
\item \label{characterizationofW} Viceversa, \textit{any stratum of codimension $r+1$, intersection of horizontal divisors, is a maximal vertical stratum not contained in any vertical divisor of $\Delta$.} By dimensional reasons, $W$ is not dominant. Moreover, if $W$ were contained in a vertical divisor of $\Delta$, there would be at least $r+2$ divisors passing though its generic point, contradicting the dlt hypothesis.
%, provided that $\mathcal{D}( \Delta)$ has maximal dimension $n$ or $Z$ is a surface or $\rho Z =1$. 
%\begin{enumerate}
%\item If $\mathcal{D}( \Delta)$ has maximal dimension, then so it is $\mathcal{D}( \Diff^*_W(\Delta))$. Hence,
%\[n-r-1= \dim\mathcal{D}( \Diff^*_W(\Delta))\leq \dim B_0 -1\leq n-r, \]
%where the first inequality follows from the fact that $\Diff^*_W(\Delta)$ is $\pi$-vertical. We conclude that $\dim B_0=n-r=\dim Z-1$.
%\item If $Z$ is a surface, the divisors which cut out $\Diff^*_W(\Delta)$ on $W$, intersecting $W$ transversely, are vertical, since all the horizontal divisors already contain $W$. This excludes that $B_0$ is a point.
%\item If $\rho(Z)=1$; see Proposition \ref{B0isadivisor}. %If $\rho Z =1$, then the canonical bundle formula can be written 
%%\[0 \sim_{\Q} K_Y+\Delta \sim_{\Q}\pi^*(K_Z + B^{=1} + J'),\]
%%where $J'$ is an effective divisor with arbitrarily small coefficient whose support does not contain $B_0$. Suppose that $B_0$ is not a divisor. As argued in \ref{characterizationofW}, it is not contained in the support of $B^{=1} + J'$. This is a contradiction, since $B_0$ would be a lc centre of the klt variety $Z$. 
%\end{enumerate}
\item \label{birationalmapW+} Let $W^+$ be a horizontal stratum of codimension $r$, which in particular contains the maximal vertical stratum $W$. Then \textit{the restriction map $$\pi|_{W^+}\colon W^+ \to Z$$ is a birational morphism.} We first claim that $\pi|_{W^+}$ is generically injective. In fact, it is generically finite, since a general fibre of $\pi|_{W^+}$ is a collection of $0$-dimensional strata of the log pair $(F_{\text{gen}}, \Delta|_{F_{\text{gen}}})$. Hence, we have that
\[
r+1=\#\{0\text{-dimensional strata in }(F_{\text{gen}}, \Delta|_{F_{\text{gen}}})\}= \sum_{W_i^+}\deg\pi|_{W_i^+},
\] 
where the first equality follows from \ref{defW} and the last summation runs over all the horizontal strata $W_i^+$ of codimension $r$. Since any $r$-uple of horizontal divisors intersects along a general fibre and there are exactly $r+1$ horizontal divisors due to \ref{defW}, there are at least $r+1={r+1 \choose r}$ such intersections; equivalently there are at least $r+1$ horizontal strata $W_i^+$. Now the previous equation forces that the $W_i^+$ are indeed $r+1$,  and $\deg\pi|_{W^+}=1$. Finally, the normality of $Z$ and Zariski's Main Theorem \cite[cor 11.4]{Hartshorne1977} implies the restriction map $\pi|_{W^+}$ is a birational morphism. In particular, $\mathcal{D}^{hor}=\mathcal{D}(\Delta|_{F_\text{gen}})$.

\noindent We remark that if $r>1$, the birationality of the morphism $\pi|_{W^+}$ follows from the weaker conclusion of \ref{defW}, i.e. $\mathcal{D}( \Delta|_{F_\text{gen}})\simeq \partial \sigma^{r}$. To this end, we are left to check that $\Delta$ has exactly $r+1$ horizontal divisor. Note that the new weak assumption implies that $\Delta$ has at most as many horizontal divisors as vertices in $\partial \sigma^{r}$, i.e. $r+1$. If $r>1$, the horizontal divisors are indeed exactly $r+1$; otherwise, there would exist an irreducible horizontal divisor $\Delta_0$ containing more than one irreducible component of $\Delta|_{F_{\text{gen}}}$, but these components are ample divisors in $F_{\text{gen}}$ and their intersection is not empty, i.e. $\Delta_0$ has a self-intersection, which contradicts the dlt assumption of $(Y, \Delta)$. %have self-intersection, thus contradicting the dlt assumption.
 On the other hand, if $r=1$, the weaker assumption is not sufficient: $\pi|_{W^+}$ can be two-to-one, even though $\mathcal{D}( \Delta|_{F_\text{gen}})\simeq \partial \sigma^{1} = \mathbb{S}^0$. For instance, consider $Y=\PP^1 \times \PP^1$ with projection $\pi: Y \to Z=\PP^1$ and $\Delta \in |\mathcal{O}(2,2)|$ a smooth element of the anti-canonical system of $Y$: despite the fact that $\mathcal{D}(\Delta|_{F_\text{gen}})\simeq \mathbb{S}^0$, the map $\pi: \Delta \to Z$ is not birational.
\item \label{imageWdivisor} \textit{The image $B_0:=\pi(W)$ is an irreducible divisor of the boundary divisor $B$} of the canonical bundle formula; see Proposition \ref{canonicalbundleformula}.(2). The restriction map $\pi|_{W^+}\colon W^+ \to Z$ is birational by \ref{birationalmapW+}. Further, $W$ is cut out on $W^+$ by one horizontal divisor, and hence it is effective and $\pi|_{W^+}$-ample. By the negativity lemma \cite[Lemma 3.39]{KollarMori1998}, $W$ cannot be contracted by the morphism $\pi$.
\item \label{propertynumberofverticallogcentre} \textit{Maximal vertical lc centres not contained in any vertical divisor of $\Delta$ are disjoint}. In view of \ref{defW}, the union of all these maximal vertical lc centres is the intersection of all the horizontal divisors. By the dlt assumption, the (disjoint) connected components of this intersection are irreducible. % (i.e. $W$ and $W'$ do not intersect) and they are at most two. The former is a consequence of the dlt assumption; see \cite[Definition 8]{deFernexKollarXu2017} or \cite[Proposition 4.16]{Kollar2013a}. The latter follows from \cite[Proposition 4.37]{Kollar2013a} applied to the logCY pair $(W^+, \Diff^*_{W^+}\Delta)$.   

\noindent Further, \emph{if $\rho (Z)=1$ and $\dim Z>1$, there exists at most one maximal vertical stratum not contained in any vertical divisor.} On the contrary, let $W$ and $W'$ be such vertical strata. Then the subvarieties $\pi(W)$ and $\pi(W')$ are ample divisors by \ref{imageWdivisor}, hence they intersects, although the lc centres $W$ and $W'$ are disjoint in $W^+$, defined as in \ref{birationalmapW+}. This contradicts Koll\'{a}r-Shokurov connectedness theorem \cite[Theorem 5.48]{KollarMori1998} over $\pi(W) \cap \pi(W')$. 
\item \label{ZB=1qdltpair} \textit{The pair $(Z, B^{=1})$ is qdlt.} The pair $(W^+, \Diff^*_{W^+}\Delta)$ and $(Z, \pi_*\Diff^*_{W^+}\Delta)$ are crepant birational logCY pairs because the morphism $\pi|_{W^+}$ is a contraction; see for instance \cite[Definition 10]{KollarXu2016}. Since $\pi_*\Diff^*_{W^+}\Delta^{=1}=B^{=1}$, the pair $(Z, B^{=1})$ is lc and its lc centres are dominated by lc centres of $\Delta$, since $(Z, B^{=1})$ is dominated by the crepant birational pair $(W^+, \Diff^*_{W^+}\Delta^{=1})$. We need to check that $(Z, B^{=1})$ is qdlt. By \cite[Proposition 5.5]{HogadiXu009} and \ref{propertynumberofverticallogcentre}, the pair $(Z,B-\sum_i\pi(W_i))$ is qdlt, where $W_i$ are all the maximal vertical lc centres not contained in a vertical divisor. % (or $(Z,B-\pi(W-I)-\pi(W'))$ is so, if $W^+$ contains two disjoint vertical lc centre); 
Without loss of generality, we can restrict our analysis to a neighbourhood of $B_0\coloneqq \pi(W_0)$. Note that there are $d-1$ components of $B-B_0$ passing through any log centre of codimension $d$ of $(B_0,\Diff_{B_0}(B))$, because the lc centres of the dlt pair $(W, \Diff^*_W(\Delta))$ are cut by vertical divisors whose images via $\pi$ are components of $B^{=1}-B_0$ (cf. Proposition \ref{canonicalbundleformula}.(2)). Adding $B_0$ itself, any log centre of codimension $d$ is contained in $d$ components of $B^{=1}$, and we conclude that $(Z, B)$ is qdlt. 
\end{enumerate}
\begin{rmk}
Note that section \S \ref{Reduction of vertical strata of maximal dimension to vertical divisors} does not rely in an essential way on the hypothesis on the Picard number of $Y$. In the following, this information will be used only to constrain the topology or the triangulation of the dual complex $\mathcal{D}(B)$. 
\end{rmk}

\subsubsection{} We prove Theorem \ref{DualComplexlogCYMorifibrespace} under the additional hypothesis of the existence of the vertical stratum $W$. 

\begin{thm}\label{generalizationthmW}
 Let $(Y, \Delta)$ be a dlt pair such that:
 \begin{enumerate}[label=(\roman*)]
 \item $Y$ is a $\Q$-factorial projective variety of dimension $n+1$;
 \item (Mori fibre space) $\pi\colon Y \to Z$ is a Mori fibre space of relative dimension $r$;
 \item (logCY) $K_Y+\Delta \sim_\Q 0$;
 \item there exists a maximal vertical stratum which is not a vertical divisor.
 \end{enumerate}
 Let $B$ be the boundary divisor given by the canonical bundle formula.
 
Then, $\mathcal{D}(\Delta)$ is PL-homemorphic to $ \mathcal{D}(B^{=1})*\mathbb{S}^{r-1}$.
\end{thm}
\begin{proof}
Let $f:Y' \to Y$ be the blow-up of $Y$ along all the maximal vertical strata not contained in any vertical divisor. Denote any of this strata by $W$, and the corresponding exceptional divisor by $E$. The dlt pair $(Y', \Delta'\coloneqq \Delta + \text{Exc})$, where $\text{Exc}$ is the sum of all the $f$-exceptional divisors, is crepant birational to $(Y, \Delta)$. 

The dual complex $\mathcal{D}(\Delta')$ is a star subdivision of $\mathcal{D}(\Delta)$ by \cite{Stepanov2006} or \cite[\S 9]{deFernexKollarXu2017}. In the same way as for $\Delta$, we distinguish between horizontal and vertical strata of $\Delta'$ with respect to the morphism $\pi \circ f$. We claim that $(Y', \Delta')$ has combinatorial product type, and that 
\[\mathcal{D}(\Delta) \simeq_{\text{PL}} \mathcal{D}(\Delta') = \mathcal{D}(B^{=1})* \mathcal{D}(\Delta|_{F_{\text{gen}}}) = \mathcal{D}(B^{=1})*\partial \sigma^{r} \simeq \mathcal{D}(B^{=1})*\mathbb{S}^{r-1}.\]

In order to prove the claim, we observe the following facts.
\begin{enumerate} %[label = (\roman*)]
\item\label{item:horstrataE} Any horizontal stratum of $\Delta'$ intersects $E$, since $W$ is contained in any horizontal stratum of $\Delta$.
\item\label{item:horstrataU} Any horizontal stratum of $\Delta'$ intersects $E$ in a unique irreducible component. Otherwise, the image of a horizontal stratum of $\Delta'$ intersecting $E$ multiple times is not normal, which contradicts the dlt assumption.
\item \label{item: producttype} The pair $(E, \Diff^*_E(\Delta'))$ has combinatorial product type with respect to the morphism $f|_E$. Indeed, the maximal vertical strata of $\Diff^*_E(\Delta')$ are vertical divisors, since all the strata of $\Delta$ transverse to $W$ are pull-back of divisors on $Z$. Further, the horizontal strata $W^+$ of codimension $r$ restricts to sections of $f|_E$, since the blow-up $f$ restricts to an isomorphism on the strict transform of $W^+$. Hence, the same argument of the first paragraphs of \S \ref{Vertical strata of maximal dimension are vertical divisors} implies that $(E, \Diff^*_E(\Delta'))$ has combinatorial product type.
\item\label{item:horstrataONCE} Intersection of horizontal strata of $\Delta'$ is horizontal. Otherwise, the vertical intersection of horizontal divisors, say $W'$, is the intersection of all ($r+1$) horizontal divisors, since the same holds for $\Delta$ by \ref{Reduction of vertical strata of maximal dimension to vertical divisors}.\ref{defW}, and since the general fibres of $\pi$ and $\pi \circ f$ are isomorphic. The image of $W'$ dominates $W$, but then $W'$ is contained in $f^{-1}(W)=E$. This is a contradiction, since the codimension-$(r+1)$ stratum $W'$ is contained in $r+2$ divisors of $\Delta'$, namely $r+1$ horizontal divisors plus $E$. 
\end{enumerate}
By item (\ref{item:horstrataE}), (\ref{item:horstrataU}) and (\ref{item: producttype}), any horizontal stratum of $\Delta'$ intersects every vertical strata in a unique connected component, and by item (\ref{item:horstrataONCE}) any intersection of horizontal divisors is horizontal. So, the pair $(Y', \Delta')$ has combinatorial product type, and  $\mathcal{D}( \Delta')= \mathcal{D}^{\text{vert}}*\mathcal{D}(\Delta'|_{F_{\text{gen}}})$. 

The pair $(W^+, \Diff^*_{W^+}(\Delta))$, as defined in \ref{Reduction of vertical strata of maximal dimension to vertical divisors}.\ref{birationalmapW+}, is crepant birational to its strict transform $(W'^{+}, \Diff^*_{W'^{+}}(\Delta'))$ via $f$, so that \[\mathcal{D}^{\text{vert}}=\mathcal{D}(\Diff^*_{W'^{+}}(\Delta'))= \mathcal{D}(\Diff^*_{W^{+}}(\Delta))=\mathcal{D}(B^{=1}),\]
where the last equality follows from \ref{Reduction of vertical strata of maximal dimension to vertical divisors}.\ref{ZB=1qdltpair}.
Further, $\mathcal{D}(\Delta'|_{F_{\text{gen}}})=\mathcal{D}(\Delta|_{F_{\text{gen}}})= \partial \sigma^r$, since the general fibres of $\pi$ and $\pi \circ f$ are isomorphic. We conclude that $\mathcal{D}(\Delta) \simeq_{\text{PL}}  \mathcal{D}(B^{=1})*\partial \sigma^{r}*\partial \sigma^{r}$, as claimed.
\end{proof}
\begin{proof}[Proof of Theorem \ref{DualComplexlogCYMorifibrespace} under the additional assumption of the existence of $W$]
%If $K_W$ is $\QQ$-linearly trivial, then $W$ is a minimal lc centre of the pair $(Y, \Delta)$, and by \cite[Proposition 4.37]{Kollar2013a}, $\mathcal{D}(B^{=1})=\mathcal{D}(\Diff^*_{W^{+}}(\Delta))$ is either one or two points. By \ref{Reduction of vertical strata of maximal dimension to vertical divisors}.\ref{propertynumberofverticallogcentre}, the latter case does not occur if $\rho(Y)=2$.

Suppose that $\rho(Y)=2$. Then by \ref{Reduction of vertical strata of maximal dimension to vertical divisors}.\ref{ZB=1qdltpair} and Theorem \ref{conjsncFanoPic1},  $\mathcal{D}(B^{=1})$ is PL-homeomorphic either to a ball $\mathbb{B}^{m}$ of dimension $m \leq {n-r}$ or to a sphere $\mathbb{S}^{n-r}$.
 Hence, Theorem \ref{generalizationthmW} implies that $\mathcal{D}(\Delta)$ is PL-homeomorphic either to a ball $\mathbb{B}^{m}$ of dimension $r \leq m \leq {n}$ or to a sphere $\mathbb{S}^{n}$.  
\end{proof}

\subsection{Maximal vertical strata are vertical divisors}\label{Vertical strata of maximal dimension are vertical divisors} Let $(Y, \Delta)$ be a dlt pair as in Theorem \ref{DualComplexlogCYMorifibrespace}. Suppose now that any maximal vertical strata is a vertical divisor. If  $(Y, \Delta)$ has combinatorial product type, then 
\[\mathcal{D}( \Delta)= \mathcal{D}^{\text{vert}}*\mathcal{D}^{\text{hor}}= \mathcal{D}( B)*\mathcal{D}( \Delta_{F_{\text{gen}}}).\]
By Theorem \ref{conjsncFanoPic1} and Remark \ref{rmk:logCYpaironthebottom}, $\mathcal{D}( \Delta)$ is then PL-homeomorphic to one of the following $\Delta$-complexes:
\begin{align*}
\sigma^k * \sigma^l & \simeq \sigma^m, \qquad \quad \,\, \text{if } m:=k+l+1 \leq n;\\
\mathbb{S}^{n-r}*\sigma^{k} & \simeq \begin{cases}
\sigma^m, \quad & \text{if }n-r+1\leq m:=n-r+k+1 \leq n;\\
\mathbb{S}^{n-r}, \quad & \text{if }\mathcal{D}^{\text{hor}}=\emptyset;
\end{cases}\\
\sigma^{k}*\mathbb{S}^{r-1} & \simeq \begin{cases}
\sigma^m, \quad & \text{if }r\leq m:=k+r \leq n;\\
\mathbb{S}^{r-1}, \quad & \text{if }\mathcal{D}^{\text{vert}}=\emptyset;
\end{cases}\\
\mathbb{S}^{n-r}*\mathbb{S}^{r-1} & \simeq \mathbb{S}^n.
\end{align*}

However, if this is not the case, by the connectivity of ample divisors the only horizontal strata which can fail to intersect a general fibre in a unique connected component have codimension $r$, namely those which restrict to points onto a general fibre. Call one of this horizontal strata $W$. The existence of $W$ implies that $\mathcal{D}( \Delta_{F_{\text{gen}}})$ has maximal dimension, and that it cannot be simplicial if $r>1$. Otherwise,  by Theorem \ref{conjsncFanoPic1}, $\mathcal{D}( \Delta_{F_{\text{gen}}})$ is either $\sigma^{r-1}$ or $ \partial \sigma^r$. %, but in both options the restriction of $\pi$ to any horizontal strata of codimension $r$ is generically injective, which contradicts the existence of $W$. 
However, the existence of $W$ implies that $\Delta$ has at least two 0-dimensional strata, which excludes $\mathcal{D}(\Delta_{F_{\text{gen}}}) \simeq \sigma^{r-1}$; while the last paragraph of \ref{Reduction of vertical strata of maximal dimension to vertical divisors}.\ref{birationalmapW+} excludes $\mathcal{D}(\Delta_{F_{\text{gen}}}) \simeq \partial \sigma^r$. Hence, by Theorem \ref{conjsncFanoPic1}, $\mathcal{D}( \Delta_{F_{\text{gen}}})$ is a non-simplicial sphere of type $\ref{conjsncFanoPic1}.(\ref{regularsphere})$, namely $\sigma^{r-1}\cup_{\partial \sigma^{r-1}}\sigma^{r-1}$. In particular, there exists exactly one horizontal stratum $W$ of codimension $r$, and  this
stratum maps generically two-to-one to $Z$ via $\pi$. The purpose of this section is to describe the PL-homeomorphism type of $\mathcal{D}(\Delta)$ assuming the existence of $W$. %Let $(Z, B'+J')$ be the generalized logCY pair given by the canonical bundle formula 
%\[0 \sim_{\Q} K_W+\Diff^*_W(\Delta) \sim_{\Q}\pi^*(K_Z + B' + J').\]

\subsubsection{} \label{boundary B empty}
If $B^{=1}=0$, then $\mathcal{D}( \Delta)=\mathcal{D}^{\text{hor}}$. Since $W$ maps generically two-to-one, the map $\text{hor}: \sigma^{r-1}\cup_{\partial \sigma^{r-1}}\sigma^{r-1}=\mathcal{D}( \Delta_{F_{\text{gen}}}) \twoheadrightarrow \mathcal{D}^{\text{hor}}$, defined in \ref{generalitiesdualcomplexpicard2}.(\ref{definition map r}), identifies the two cells of maximal dimensions of the domain. Hence, $\mathcal{D}( \Delta)=\sigma^{r-1}$. Therefore, in the following we can suppose that $B^{=1}\neq 0$.

\subsubsection{} Note that also in this case the pair $(Z, B^{=1})$ is qdlt. Write $\Delta = \Delta^{\text{vert}} + \Delta^{\text{hor}}$ such that the support of $\Delta^{\text{vert}}$ is vertical, and that of $\Delta^{\text{hor}}$ horizontal. Up to tensoring with the pull-back via $\pi$ of an ample divisor in $Z$, we can choose a general ample $\Q$-divisor $H$ in $Y$ with small coefficients such that $H \sim_{\pi} \Delta^{\text{hor}}$. The canonical bundle formula for the new pair $(Y,\Delta^{\text{vert}} + H)$ still holds (cf. \cite[\S 20]{KollarXu2016}), and it yields the same boundary part $B^{=1}$  by \ref{canonicalbundleformula}.(2) (but $B^{<1}$ may change). Therefore, \cite[Proposition 5.5]{HogadiXu009} implies that $(Z, B^{=1})$ is qdlt.

\subsubsection{} \label{boundary B simplex}
Now, if the branch locus of $\pi|_W$ has an
irreducible component $B_0$ of codimension one (e.g. when $Z$ is smooth, by the purity of the branch locus), then $(Y, \Delta + \pi^*B_0)$ is not lc at the generic point of $\pi^*B_0$, which forces $\operatorname{coeff}_{B}(B_0)<1$ by Proposition \ref{Canonical bundle formula}. As %$B^{=1}=B^{\prime =1}$ and $B^{\prime <1}\neq 0$, then 
$B^{<1}\neq 0$, $\mathcal{D}( B)$ is a standard simplex $\sigma^{k}$ by Remark \ref{boundary<1Pic1}. The dual complex $\mathcal{D}( \Delta)$ is no more the join of the dual complex of the vertical and horizontal strata, but we still have a PL-map induced by $\text{hor}$ \[\sigma^{k}*\mathcal{D}( \Delta_{F_{\text{gen}}})=\sigma^{k}* (\sigma^{r-1}_{(1)}\cup_{\partial \sigma^{r-1}}\sigma^{r-1}_{(2)})\twoheadrightarrow \mathcal{D}( \Delta),\] which identifies the vertical strata contained in $W$ passing through the two 0-dimensional strata of $(W, \Diff^*_W(\Delta))$, namely $\partial \sigma^{k} * \sigma^{r-1}_{(1)}$ and $\partial \sigma^{k} * \sigma^{r-1}_{(2)}$. As a result, we obtain 
\begin{align*}
\mathcal{D}( \Delta) & = \sigma^{k}* (\sigma^{r-1}_{(1)}\cup_{\partial \sigma^{r-1}}\sigma^{r-1}_{(2)}) \mod \partial \sigma^{k} * \sigma^{r-1}_{(1)}\equiv\partial \sigma^{k} * \sigma^{r-1}_{(2)}\\
& = (\sigma^{k}* \sigma^{r-1}_{(1)})\cup_{\sigma^{k}*\partial \sigma^{r-1}\cup \partial \sigma^{k} * \sigma^{r-1}}(\sigma^{k}* \sigma^{r-1}_{(2)})\\
& = (\sigma^{k}* \sigma^{r-1}_{(1)})\cup_{\partial (\sigma^{k}*\sigma^{r-1})}(\sigma^{k}* \sigma^{r-1}_{(2)})\simeq \mathbb{S}^m, \quad \text{with } r-1 \leq m \leq n.
\end{align*}

\subsubsection{}\label{boundary B sphere} 
Without loss of generality, in the following we further assume that $\mathcal{D}( \Delta)$ has maximal dimension and $\Delta^{=1}=\Delta$. Indeed, as we remarked at the beginning of \S \ref{Vertical strata of maximal dimension are vertical divisors},  the existence of $W$ implies that $\mathcal{D}( \Delta|_{F_{\text{gen}}})$ has maximal dimension and it is homeomorphic to a sphere, so that $\Delta|_{F}^{=1}= \Delta|_{F}$ by Remark \ref{boundary<1Pic1}. Hence, if $\mathcal{D}( \Delta)$ does not have maximal dimension, then $\mathcal{D}( B)$ does not either, and by Theorem \ref{conjsncFanoPic1} it is a standard simplex or empty. Similarly, if $\Delta^{<1}\neq 0$, then $B^{<1}\neq 0$, and $\mathcal{D}( B)$ is a standard simplex by Remark \ref{boundary<1Pic1}. In both case, we can describe $\mathcal{D}( \Delta)$ as in \ref{boundary B empty} and \ref{boundary B simplex}. 

Therefore, we are left to consider the following case:
\begin{enumerate}
\item $\mathcal{D}( \Delta)$ has maximal dimension;
\item $\Delta^{=1}=\Delta$ and $B^{=1}=B \neq 0$;
\item the ramification locus of $\pi|_W$ has codimension $\geq 2$.
\end{enumerate}
We claim that this option cannot occur, by showing that $\pi|_{W}$ induces an identification between $\mathcal{D}( \Diff^*_W(\Delta))$ and $\mathcal{D}( B)$. Indeed, this %implies that any horizontal strata intersect a general fibre in a unique connected component, which 
is a contradiction, since the fibre via $\pi|_{W}$ of any 0-dimensional stratum of $(Z, B)$ is not connected (see next item \ref{pair is qdlt}), so that the PL-morphism induced by $\pi|_{W}$ cannot be a homeomorphism. % or equivalently that $(Y, \Delta)$ has combinatorial product type, and so that \[\mathcal{D}( \Delta)= \mathcal{D}^{\text{vert}}*\mathcal{D}^{\text{hor}}\simeq \mathbb{S}^n.\]

In order to prove the claim, consider first the Stein factorization of the restriction morphism $\pi|_W$
\[
 \xymatrix{
 & (W, \Diff^*_W(\Delta)) \ar@{->}[ld]_{\varrho} \ar@{->}[d]^{\pi|_W}   \\
  (Z^{\nu}, B^{\nu}:=\nu^*(B)) \ar@{->}[r]^{\qquad \nu}_{ \qquad 2:1}&(Z,  B).}
\]
In particular, the following facts hold.
\begin{enumerate}[label = (\roman*)]
\item \label{pair is qdlt}\textit{The pair $(Z^{\nu}, B^{\nu})$ is lc logCY}, since the morphism $\nu$ is quasi-\'{e}tale, \textit{and qdlt}. Indeed, any lc centre of codimension $d$ is the intersection of exactly $d$ components of $B^{\nu}$: such lc centre dominates a lc centre of $(Z, B)$, so it is cut by at least $d$ divisors in $B^{\nu, =1}$, and it is dominated by a lc centre of $(W, \Diff^*_W(\Delta)= \pi^{-1}|_W(B))$, so it is cut by at most $d$ divisors in $B^{\nu, =1}$. % one argue as in $\ref{strata etale}$,  exploiting the fact that $(Z^{\nu}, B^{\nu})$ sits in-between the qdlt pairs $(W, \Diff^*_W(\Delta)= \pi^{-1}|_W(B))$ and $(Z, B)$. %there exists only  To this end, note that for any the morphism $\nu$ is \'{e}tale at the generic point of the lc centres of $B$. Otherwise the branch locus would contain a fortiori a 0-dimensional lc centre, but this is a contradiction by \ref{strata etale}. Hence, the statement follows from the fact that the pair $(Z, B)$ is qdlt by \cite[Proposition 5.5]{HogadiXu009}.
\item \textit{The lc centres of $(Z^{\nu}, B^{\nu})$ are sent to lc centres of $(Z,B)$ via $\nu$} by \cite[Proposition 5.20]{KollarMori1998}. \textit{It means that there exists a PL-map 
\[\nu_*\colon \mathcal{D}( B^{\nu})\to\mathcal{D}( B).\]}
\item \label{alphafinite} \textit{The morphism $\nu_*\colon \mathcal{D}( B^{\nu})\to\mathcal{D}( B)$ has finite fibres of cardinality at most two}, since the map $\nu$ is finite, and any lc centre of $(Z, B)$ is dominated by lc centres of the same dimension in $(Z^{\nu}, B^{\nu})$. %by \ref{strata etale}.
\item \label{alphaetale} \textit{The morphism $\nu_*\colon \mathcal{D}( B^{\nu})\to\mathcal{D}( B)$ is a topological covering map.} Let $v_d$ be a cell of maximal dimension, say $d$, in the ramification locus of $\nu_*$. By maximality, $\nu_*$ is a topological covering of degree two onto $\Link(v_d, \mathcal{D}( B))\simeq \mathbb{S}^{\dim Z-d-2}$. We claim that the existence of $v_d$ yields a contradiction, so that $\nu_*$ must be a topological covering map.
\begin{enumerate}
\item If $\dim Z-d-2>1$ or $\dim Z-d-2=0$, the link is simply-connected or the union of two points, so that \[\Link(\nu_*^{-1}(v_d), \mathcal{D}( B^{\nu}))\simeq \mathbb{S}^{\dim Z-d-2}\sqcup \mathbb{S}^{\dim Z-d-2}.\] This is a contradiction by the connectedness theorem \cite[Proposition 4.37]{Kollar2013a}. %, since the maximality of $\mathcal{D}( B^{\nu})$ implies that the link of any of its cells of dimension $\geq 1$ is connected (cf. \cite[\S 32]{KollarXu2016}).
\item \label{surfacecaselink} If $\dim Z-d-2=1$, then $\Link(v_d, \mathcal{D}( B))$ is isomorphic to a circle with two or three vertices, due to Theorem \ref{conjsncFanoPic1}. Comparing with the previous connectivity argument, we can suppose that $\Link(\nu_*^{-1}(v_d), \mathcal{D}(B^{\nu}))$ is a connected topological covering space of $\Link(v_d, \mathcal{D}( B))$. Therefore, the $\Link(\nu_*^{-1}(v_d), \allowbreak \mathcal{D}( B^{\nu}))$ is a circle with four or six vertices.

This implies that there exist two 2-dimensional lc centres 
\begin{align*}
&(W,  \Diff^*_W(B):=\sum_i B_i)\\
&(W^{\nu},  \Diff^*_{W^\nu}(B^\nu):=\sum_i B^\nu_{i, 0} + \sum_i B^{\nu}_{i,1})
\end{align*} 
of the log pairs $(Z, B)$ and $(Z^\nu, B^\nu)$ respectively, such that:
\begin{enumerate}[label = (\arabic*)]
\item $W^{\nu}=\nu^{-1}(W)$;
\item $\Link(v_d, \mathcal{D}( B))= \mathcal{D}( \Diff^*_{W}(B))$;
\item $\Link(\nu_*^{-1}(v_d), \mathcal{D}( B^{\nu}))= \mathcal{D}(\Diff^*_{W^{\nu}}(B^{\nu}))$.
\end{enumerate}
\begin{center}
\tikzstyle{vertex}=[circle, draw, fill=black!50,inner sep=0pt, minimum width=2.5pt]
\begin{tikzpicture}[thick, scale = 1.2]%
\draw (-2,0) circle (0.5 cm);
\draw (-2,0.5) node[vertex, label=above:$B^{\nu}_{1,0}$]{};
\draw (-1.5,0) node[vertex, label=right:$B^{\nu}_{2,0}$]{};
\draw (-2,-0.5) node[vertex, label=below:$B^{\nu}_{1,1}$]{};
\draw (-2.5,0) node[vertex, label=left:$B^{\nu}_{2,1}$]{};
\node[align=right, below] at (2,-1)%
{$\mathcal{D}( \Diff^*_{W}(B))$};
\draw (2,0) circle (0.5 cm);
\draw (2,0.5) node[vertex, label=above:$B_1$]{};
\draw (2,-0.5) node[vertex, label=below:$B_2$]{};
\node[align=right, below] at (-2,-1)%
{$\mathcal{D}( \Diff^*_{W^\nu}(B^\nu))$};
\end{tikzpicture}
\end{center}
Note that $B_1$ is cut out on $W$ by an ample divisor of $Z$, as $\rho(Z)=1$. Hence, the curve $\nu^*(B_1)=B^{\nu}_{1,0}+B^{\nu}_{1,1}$ is a non-connected ample divisor, which yields a contradiction by the Hodge index theorem.
\end{enumerate}
\end{enumerate}

All these facts imply that 
$\mathcal{D}( B^{\nu})= \mathcal{D}( B)$. Indeed, if $\dim Z>2$, then $\mathcal{D}( B)$ is a simply-connected sphere. Hence, $\nu_*$ has degree one and $\mathcal{D}( B^{\nu})= \mathcal{D}( B)$. Instead, if $\dim Z=2$ and $\nu_*\colon \mathcal{D}( B^{\nu})\to\mathcal{D}( B)$ is a non-trivial covering map, then as before the pull-back of an ample boundary divisor is not connected, which yields a contradiction by the Hodge index theorem (cf. \ref{surfacecaselink}). Clearly, the possibility $\dim Z=1$ never occur: in this case $Z$ would be a rational curve, and there are no quasi-\'{e}tale morphism with target $\PP^1$. %As a corollary of \ref{alphafinite} and \ref{alphaetale}, we conclude that $\mathcal{D}( B^{\nu})= \mathcal{D}( B)$. 

Finally, observe that $\mathcal{D}( \Diff^*_W(\Delta))= \mathcal{D}( B^{\nu})$. 
Indeed, write $B^{\nu}= \sum_{i \in I} B^{\nu}_i$ and $\Diff^*_W(\Delta)= \sum_{i \in I}\Delta^{\text{vert}}_j|_{W}$. Consider now the correspondence 
\begin{align*}
\text{c.c. } \bigcap_{j \in J} \Delta^{\text{vert}}_j|_{W} & \rightarrow \text{c.c. } \bigcap_{j \in J} \rho(\Delta^{\text{vert}}_j|_{W})\\
\text{c.c. } \bigcap_{j \in J} \rho^{-1} B^{\nu}_{j} & \leftarrow \text{c.c. } \bigcap_{j \in J} B^{\nu}_{j},
\end{align*}
where $c.c.$ stands for connected components, and $J$ runs over all the subsets of $I$, which parametrizes the irreducible components of $\Diff^*_W(\Delta)$, equivalently of $B^{\nu}$. Since any irreducible component of $\Diff^*_W(\Delta)^{=1}$ is the pull-back of a divisor in $B^{\nu}$, and because the morphism $W \to Z^{\nu}$ has connected fibres, this correspondence is a bijection between the strata of $\Diff^*_W(\Delta)$ and those of $B^{\nu}$. 

Combining the output of the last two paragraphs, we conclude that $\pi|_{W}$ induces a PL-homeomorphism between $\mathcal{D}( \Diff^*_W(\Delta))$ and $\mathcal{D}( B)$.
   
\section{Dual complex of logCY pairs on Mori fibre spaces over a surface}\label{dimension2case} In this section we adapt the arguments of \S \ref{Reduction of vertical strata of maximal dimension to vertical divisors} and \S \ref{Vertical strata of maximal dimension are vertical divisors} to describe the dual complex of logCY pairs on Mori fibre spaces whose base has dimension two.
    
    \begin{thm}[$\dim Z=2$]\label{dualcomplexsurface}
      Let $(Y, \Delta)$ be a dlt pair such that:
      \begin{enumerate}[label=(\roman*)]
      \item $Y$ is a $\Q$-factorial projective variety of dimension $n+1$;
      \item (Mori fibre space) $\pi\colon Y\to Z$ is a Mori fibre space of relative dimension $r$ onto a %$\Q$-factorial 
      projective surface $Z$;
      \item (logCY) $K_Y+\Delta \sim_\Q 0$.
      \end{enumerate}
      Then $\mathcal{D}( \Delta)$ has one of the following PL-homeomorphism types:
      \begin{enumerate}
      \item $\mathbb{S}^1, \mathbb{S}^{n-2}, \mathbb{S}^{n-1}$ or $\sigma^{m}$ with $m < n$ if $(Y, \Delta)$ has not maximal intersection;
      \item $\sigma^n, \mathbb{S}^n$ or $ \PP^2(\R)* \mathbb{S}^{n-3}$ if $(Y, \Delta)$ has maximal intersection. %, if $\Delta=\Delta^{=1}$, $(X, \Delta)$ has maximal intersection and $K_Y + \Delta$ is not Cartier.
     % \item \label{surfacecasesphere} $\mathcal{D}( \Delta) \simeq \mathbb{S}^n$. 
      \end{enumerate}
    \end{thm}
    \begin{proof}
    We distinguish two cases: either any maximal vertical stratum is a vertical divisor or not.
    
    \textbf{Step 1.} First, suppose that there exists a vertical stratum of codimension $n$ not contained in any vertical divisor of $\Delta$. The properties \ref{Reduction of vertical strata of maximal dimension to vertical divisors}.\ref{defW}-\ref{ZB=1qdltpair} continue to hold. Let $W^+$ be a horizontal stratum of dimension two. Then $\mathcal{D}(B^{=1})=\mathcal{D}(\Diff^*_{W^+}(\Delta))$ by \ref{Reduction of vertical strata of maximal dimension to vertical divisors}.\ref{ZB=1qdltpair}. In particular, $\mathcal{D}(B^{=1})$ is the dual complex of a logCY surface with at least one vertex, so it is a point, a union of two points, a segment or a circle; see for instance \cite{KollarXu2016}. Finally, Theorem \ref{generalizationthmW} implies that $\mathcal{D}(\Delta)$ is PL-homeomorphic to $\sigma^{n-1}$, $\mathbb{S}^{n-1}$, $\sigma^{n}$ or $\mathbb{S}^{n}$, respectively.

    \textbf{Step 2.} Suppose now that all the vertical strata are vertical divisors. 
    The only places in \S \ref{Vertical strata of maximal dimension are vertical divisors} where we have exploited the hypothesis of $\rho(Z)=1$ is to prescribe the homeomorphism type of the dual complex $\mathcal{D}^{\text{vert}}=:\mathcal{D}( B^{=1})$ and to exclude the case detailed in \ref{boundary B sphere}. 
    Under the hypothesis $\dim Z =2$, we can prescribe again the homeomorphism type of $\mathcal{D}^{\text{vert}}$, but we cannot avoid the latter issue anymore. 
    
    We first show that $\mathcal{D}^{\text{vert}}$ is empty, a point, a union of two points, a segment or a circle. Equivalently, $\mathcal{D}^{\text{vert}}$ is a manifold, eventually with boundary, of dimension at most one, either connected or union of two points. This would follow immediately from \cite[Proposition 4.37]{Kollar2013a} if $\mathcal{D}^{\text{vert}}$ were the dual complex of a logCY pair. It is indeed the case if there are no horizontal divisors in $\Delta$, so that $\mathcal{D}^{\text{vert}}=\mathcal{D}(\Delta)$. Otherwise, consider a minimal horizontal stratum of $\Delta$, denoted $W$. 
    %We warn the reader that in the previous part of the proof, the symbol $W$ denoted instead a maximal vertical stratum. Overwriting this symbol should not cause any confusion, since  such maximal vertical strata do not appear under the new assumption that all vertical strata are contained in vertical divisors. Further, the choice will help to draw the analogy between this part and \S \ref{Vertical strata of maximal dimension are vertical divisors} even at a notational level.     
    Note that $\Diff^*_{W}(\Delta)^{=1} = \pi|_{W}^* B^{=1}$, i.e. the components of $\Diff^*_{W}(\Delta)^{=1}$ are cut by vertical divisors by the minimality of $W$. Hence, any component of $B^{=1}$ can intersect at most two other components, since the same occurs for the logCY pair $(W, \Diff^*_{W}(\Delta))$ by \cite[Proposition 4.37]{Kollar2013a}. For a graph like $\mathcal{D}^{\text{vert}}$, this implies that $\mathcal{D}^{\text{vert}}$ is a manifold with boundary. A similar argument proves the statement about the connectedness of $\mathcal{D}^{\text{vert}}$, by applying \cite[\S 16]{KollarXu2016} to the pair $(W, \Diff^*_{W}(\Delta))$.
%    The only places in \S \ref{Vertical strata of maximal dimension are vertical divisors} where we have exploited the hypothesis $\rho(Z)=1$ is either to grant the divisor $B+J$ is a boundary or to exclude case \ref{Vertical strata of maximal dimension are vertical divisors}.\ref{surfacecaselink}. By dimensional reason, the dual complex $\mathcal{D}^{\text{vert}}=:\mathcal{D}( B^{=1})$ is well-defined independently of the singularities of the pair $(Z, B^{=1})$. However, we cannot avoid the second issue.
   
   As a result, if $(Y, \Delta)$ has combinatorial product type, the description of $\mathcal{D}(\Delta)$ is identical to that of the case of $\rho(Y)=2$; see first paragraphs of \S \ref{Vertical strata of maximal dimension are vertical divisors}. 
    
    \textbf{Step 3.} Suppose now that $(Y, \Delta)$ has not combinatorial product type. As in \S \ref{Vertical strata of maximal dimension are vertical divisors}, there exists exactly one horizontal stratum $W$ of codimension $n$, and  this stratum maps generically two-to-one to $Z$ via $\pi$.
    Without loss of generality, we can assume $B^{=1} \neq 0$; otherwise the same argument of \S \ref{boundary B empty} applies. Since the components of $\Diff^*_{W}(\Delta)^{=1}$ are restrictions of vertical divisors, and the pair $(W, \Diff^*_{W}(\Delta))$ is dlt, then the map $\pi|_W$ is finite at the 0-dimensional strata of $(W, \Diff^*_{W}(\Delta))$. In particular, the image of any stratum is a lc centre of $B^{=1}$ of the same dimension. Therefore, there exists a surjective PL-morphism 
    \[
    \pi_*:\mathcal{D}(\Diff^*_{W}(\Delta)) \twoheadrightarrow  \mathcal{D}(B)=\mathcal{D}^{\text{vert}}
    \]
    whose fibres have cardinality at most two. If $\pi_*$ is a homeomorphism, then $(Y, \Delta)$ has combinatorial product type. So by assumption, there should exist at least a point in $\mathcal{D}^{\text{vert}}$ whose fibre via $\pi_*$ consists of two points. Note that an edge of $\mathcal{D}(\Diff^*_{W}(\Delta))$ cannot be contained in the ramification locus of $\pi$; %, unless $\pi_*$ is a homeomorphism; 
    otherwise, there would exist an edge in the ramification locus adjacent to an edge which does not belong to it, but then the link of the common vertex would have at least three points, which is a contradiction. For the same reason, a vertex is in the branch locus only if it is a boundary of $\mathcal{D}(B)$. Because of Step 2 and the previous remarks, $\pi_*$ is one of the following map (up to isotopy):
    \begin{enumerate}
    \item $\mathbb{S}^0 \to \sigma^0$;
    \item $\sigma^1 \simeq [-1,1] \to \sigma^1 \simeq [0,1]$, $x \mapsto |x|$;
     \item $\mathbb{S}^1 \to \sigma^1 \simeq [-1,1]$, $ z \mapsto \Re(z)$ with $z \in \mathbb{S}^1 \subset \C$;
    \item $\mathbb{S}^1 \to \mathbb{S}^1$, $z \mapsto z^2$.
    \end{enumerate}
    
    In case (1) and (3), the PL-homeomorphic type of $\mathcal{D}(\Delta)$ can be described as in \S \ref{boundary B simplex}; the triangulation may be finer (for instance if $\mathcal{D}(B)$ is a segment with more than two vertices), but the description of $\mathcal{D}(\Delta)$ is identical. 
    
    Case (2) is similar, but for completeness we describe its triangulation. Denote by $B_0$ the components of $B^{=1}$ corresponding to $\{0\}=\sigma^1(0) \in [0,1] \simeq \sigma^1$. Away from the ramification (more precisely, over the formal completion of $Z$ along $B - B_0$), the pair $(Y, \Delta)$ has combinatorial product type and its dual complex is \[\mathcal{D}^{\text{vert}}*\mathcal{D}(\Delta_{F_{\text{gen}}}) =\sigma ^{1}*\mathcal{D}( \Delta_{F_{\text{gen}}})=\sigma^{1}* (\sigma^{n-2}_{(1)}\cup_{\partial \sigma^{n-2}}\sigma^{n-2}_{(2)}).\]
    The presence of the ramification suggest that the cell associated to $W$ and $\pi^*|_W B_0$ are counted twice in this join. In simplicial term, there exists a surjective PL-morphism $\mathcal{D}^{\text{vert}}*\mathcal{D}(\Delta_{F_{\text{gen}}}) \to \mathcal{D}(\Delta)$ which induces the following identifications:
    \[
    \emptyset * \sigma^{n-2}_{(1)} \sim \emptyset * \sigma^{n-2}_{(2)} \qquad  \sigma^1(0)* \sigma^{n-2}_{(1)} \sim \sigma^1(0)* \sigma^{n-2}_{(2)}.
    \] 
    Note that
    \[
    \partial (\sigma^1 * \sigma^{n-2})= \partial \sigma^1 * \sigma^{n-2} \cup \sigma^1 * \partial \sigma^{n-2} =
    (\sigma^1(0)* \sigma^{n-2} \cup \sigma^1 * \partial \sigma^{n-2}) \cup \sigma^1(1)* \sigma^{n-2},
    \]
    and we have
    \begin{align*}
    \mathcal{D}( \Delta) & = \sigma^{1}* (\sigma^{n-2}_{(1)}\cup_{\partial \sigma^{n-2}}\sigma^{n-2}_{(2)})/ \sim\\
    & = (\sigma^{1}* \sigma^{n-2}_{(1)})\cup_{\sigma^1(0)* \sigma^{n-2} \cup \sigma^1 * \partial \sigma^{n-2}}(\sigma^{1}* \sigma^{n-2}_{(2)})\\
    & = (\sigma^{1}* \sigma^{n-2}_{(1)})\cup_{\partial (\sigma^{1}*\sigma^{n-1}) \setminus \sigma^1(1)* \sigma^{n-2}}(\sigma^{1}* \sigma^{n-2}_{(2)})\simeq \mathbb{S}^{n}\setminus \sigma^{n-1} \simeq \sigma^n.
    \end{align*}

    \textbf{Step 4.} Consider now case (4). It means that there exists exactly one minimal horizontal stratum $W$ such that the map $\pi_*:\mathcal{D}(\Diff^*_{W}(\Delta))\simeq \mathbb{S}^1 \twoheadrightarrow  \mathcal{D}(B)\simeq \mathbb{S}^1$ is a topological covering of degree two. 
    
    We claim that the assumption forces the ramification of $\pi|_W$ to have codimension $\geq 2$. If not, we can write $K_W = \pi^*(K_Z + \frac{1}{2}\Br(\pi|_W))$. By Remark \ref{boundary<1Pic1}, the homeomorphism $\mathcal{D}(\Diff^*_{W}(\Delta))\simeq \mathbb{S}^1$ implies 
    $\Diff^*_{W}(\Delta)^{=1}=\Diff^*_{W}(\Delta) \sim_{\QQ} - K_W$. So, since $\Diff^*_{W}(\Delta)^{=1}=\pi|_W^*B^{=1}$, we conclude that $K_Z+B^{=1}+ \frac{1}{2}\Br(\pi|_W)$ is numerically trivial, and $\Q$-linearly trivial, as $Z$ is a $\Q$-factorial rational surface with rational singularities; see \cite[Proposition 3.36]{KollarMori1998}, \cite[\S 18]{KollarXu2016},  and \cite[Proposition 5.5]{HogadiXu009}. Therefore, as in Remark \ref{boundary<1Pic1}, we have   
   \begin{align*}
    h_{1}(\mathcal{D}( B^{=1}), \C) & =h^1(Z, \OO_{B^{=1}})=h^{2}(Z, \OO(-B^{=1}))\\
    & =h^{0}(Z, K_Z+B^{=1})= h^0(Z, - \frac{1}{2}\Br(\pi|_W))=0.
   \end{align*}
    The first equality follows from the fact that $B^{=1}$ is a collection of rational curves meeting transversely. The second is a consequence of the vanishing $H^i(Z, \mathcal{O}_Z)=0$ for $i>0$,  due to the fact that $Z$ is a rational surface with rational singularities. Finally, the third equality is Serre's duality. All together these equalities yield a contradiction, since $h_{1}(\mathcal{D}( B^{=1}), \C)=h_{1}(\mathbb{S}^1, \C)=1$.
     
    Hence, we can reduce to the case that the ramification locus of $\pi|_W$ has codimension $\geq 2$. Consider the Stein factorization of the restriction morphism $\pi|_W$ and the cartesian diagrams given respectively by the pairs of morphism $(\pi|_W, \nu)$ and $(\pi, \nu)$:
    \[
     \xymatrix{
     (W^{\nu}:=W \times_Z Z^{\nu}, \nu^{\prime\, *}_{W^{\nu}}\Diff^*_W(\Delta)=\pi^{\prime\,} |_W^*(B^{\nu})) \ar@{->}[d]_{\pi'|_{W^{\nu}}} \ar@{->}[r]^{\qquad \qquad \qquad \quad\nu'_{W^{\nu}}}& (W, \Diff^*_W(\Delta)) \ar@{->}[ld]_{\varrho} \ar@{->}[d]^{\pi|_W}   \\
      (Z^{\nu}, B^{\nu}:=\nu^*(B^{=1})) \ar@{->}[r]^{\qquad \nu}_{ \qquad 2:1}&(Z,  B^{=1}).}
    \]
    \[
    \xymatrix{
    Y^{\nu}:=Y \times_Z Z^{\nu} \ar@{->}[d]_{\pi'} \ar@{->}[r]^{\qquad \quad \nu'} & Y \ar@{->}[d]^{\pi}\\
    Z^{\nu} \ar@{->}[r]_{\nu} & Z.
    }
    \]
  Note that the equality $\Diff^*_{W}(\Delta)= \Diff^*_{W}(\Delta)^{=1}=\pi|_W^*B^{=1}$ implies $\nu^{\prime\, *}_{W^{\nu}}\Diff^*_W(\Delta)=\pi^{\prime\,} |_W^*(B^{\nu})$. Further, since $\nu$ is quasi-\'{e}tale, the log pair $(Y^{\nu}, \nu^{\prime, \, *}\Delta)$ %$(W^{\nu}, \nu^{\prime \, *}\Diff^*_W(\Delta))$ 
  is logCY, but in general it fails to be dlt over the ramification locus of $\nu$. One can check it for instance from the computation for $W^{\nu}$ below. As remarked in Step 2 and since the PL-map
    $\nu_*\colon \mathcal{D}( B^{\nu}) \to \mathcal{D}( B)$
    is a topological covering of degree at most two, no 0-dimensional stratum of $(Z, B)$ is contained in the branch locus of $\nu$, denoted $\Br(\nu)$. In particular, up to shrinking $Z$, namely removing $\Br(\nu)$, we can suppose that $\nu$ is \'{e}tale. In fact, the dual complexes are unchanged:
    \begin{align*}
    \mathcal{D}(B|_{Z\setminus \Br(\nu)}) & = \mathcal{D}( B);\\
    \mathcal{D}(B^{\nu}|_{Z^{\nu}\setminus \nu^{-1}\Br(\nu)}) & = \mathcal{D}( B^{\nu});\\
    \mathcal{D}( \Diff^*(\Delta)|_{W\setminus  \pi^{\prime \, -1}(\Br(\nu))}) & = \mathcal{D}( \Diff^*(\Delta)).
    \end{align*}
    The advantage is that now $(Y^{\nu}, \nu^{\prime \, *}\Delta)$ is dlt, as pullback of a dlt pair via the \'{e}tale morphism $\nu'$, and the dual complex $\mathcal{D}( \nu^{\prime \, *}\Delta)$ is well-defined. 
    
    Denote by $\tau$ the involution of the covering map $\nu'$. Observe that over the generic point of $Z^{\nu}$
    \begin{align*}
    W^{\nu} & =\{(w,z)\in W \times Z^{\nu}|\, \pi(w)=\nu(\varrho(z))=\nu(z)\}\\
    & =\{(w,\varrho(w))|\, w\in W\} \sqcup \{(w,\tau \varrho(w))|\, w\in W\},
    \end{align*} 
    \emph{i.e.} $ W^{\nu}$ is the union of two strata of  $\nu^{\prime \, *}\Delta$ in $Y^{\nu}$, %irreducible components of the boundary $\nu^{\prime \, *}\Diff^*(\Delta)$, 
    exchanged by the involution $\tau$. Now, all the maximal vertical strata of $(Y^{\nu}, \nu^{\prime \, *}\Delta)$ are divisors, and the horizontal strata of codimension $r$ maps generically injective onto the base:  $\mathcal{D}(\nu^{\prime \, *}\Delta)$ has combinatorial product type. In particular, %By the first paragraph of \S \ref{Vertical strata of maximal dimension are vertical divisors}, 
    we have
    \[\mathcal{D}(\nu^{\prime \, *}\Delta)=\mathcal{D}^{\text{vert}}*\mathcal{D}^{\text{hor}}\simeq \mathbb{S}^1 * \mathbb{S}^{n-1}.\]  
    The involution $\tau$ descends to a PL-involution $\tau_*$ of $\mathcal{D}(Y^{\nu}, \nu^{\prime \, *}\Delta)$: $\tau_*$  fixes the horizontal strata of codimension $\leq n$, which form a subcomplex isomorphic to $\mathbb{S}^{n-3}$. We conclude that
   \begin{align*}
    \mathcal{D}( \Delta)& = \mathcal{D}^{\text{vert}}*\mathcal{D}^{\text{hor}}/\tau_*\simeq \mathbb{S}^1 * \mathbb{S}^{n-1}/\tau_*\\
    & = \mathbb{S}^1*\mathbb{S}^0*\mathbb{S}^{n-3}/\tau_* \simeq \mathbb{S}^2*\mathbb{S}^{n-3}/\tau_* \simeq \PP^2(\R)*\mathbb{S}^{n-3}.
   \end{align*} 
    \end{proof}
    
    \begin{cor}\label{cohomological charcterization dual complex}
    In the hypothesis of Theorem \ref{dualcomplexsurface}, $\mathcal{D}( \Delta) \simeq \mathbb{S}^n$ if and only if $(Y, \Delta)$ has maximal intersection, $\Delta=\Delta^{=1}$ and $K_Y + \Delta$ is Cartier.
    \end{cor}
    \begin{proof}
        The conditions are necessary. The argument in Remark \ref{boundary<1Pic1} implies that if $\mathcal{D}( \Delta) \simeq \mathbb{S}^n$, then $\Delta=\Delta^{=1}$ and $K_Y + \Delta \sim \mathcal{O}_Y$, thus Cartier.  
              
        The conditions are also sufficient. Indeed, by Theorem \ref{dualcomplexsurface} the maximality implies that $\mathcal{D}( \Delta)$ is isomorphic to $\sigma^n$ or $\PP^2(\R)* \mathbb{S}^{n-3}$ or $\mathbb{S}^{n}$. However, as in Remark \ref{boundary<1Pic1}, we have \[h_{n}(\mathcal{D}( \Delta), \C)=h^{0}(Y, K_Y+\Delta^{=1})\neq 0,\]
    which excludes the first two cases. 

    \end{proof}
    \begin{rmk}
    Without assuming that $Y$ carries a structure of Mori fibre space over a surface (or a curve), the additional hypotheses of Corollary \ref{cohomological charcterization dual complex}, namely $\Delta=\Delta^{=1}$, $(Y, \Delta)$ has maximal intersection and $K_Y + \Delta$ is Cartier, do not imply that $\mathcal{D}( \Delta)\simeq \mathbb{S}^n$. Indeed, we cannot exclude the occurrence of finite quotient of spheres. See also \cite[\S 33, n=4]{KollarXu2016}.
    
    We describe an example of this phenomenon. Let $Y':=\PP^1\times \PP^1\times \PP^1\times \PP^1$ and $\tau$ the involution which swaps the homogeneous coordinates $[x:y]\to[y:x]$ on each factor. Let $(Y, \Delta)$ be the dlt logCY pair where $Y:=Y'/\tau$, and $\Delta$ is the pushforward of the toric boundary of $Y'$. The dual complex $\mathcal{D}( \Delta)$ is isomorphic to $\PP^3(\R)$; see also \ref{examplePR2Sn-3} for detailed computations. However, $\Delta=\Delta^{=1}$ and $(Y, \Delta)$ has maximal intersection by construction and also $K_Y+\Delta$ is Cartier. Indeed, notice that the singularities of $Y$ are quotient singularities of type $\frac{1}{2}(1,1,1,1)$, thus cones over the projective space $\PP^3$ polarized with the line bundle $\mathcal{O}_{\PP^3}(2)$. Since the canonical class of $\PP^3$ is a multiple of the polarization, the singularities are Gorenstein; see \cite[Propostion 3.14]{Kollar2013a}. Together with the fact that $\Delta$ is supported on the smooth locus of $Y$, we conclude that $K_Y+\Delta$ is Cartier.
    \end{rmk}
\begin{exam}\label{examplePR2Sn-3}
We construct examples of dlt logCY pairs $(Y, \Delta)$ in any dimension with the following properties:
\begin{enumerate}
      \item $Y$ is a $\Q$-factorial projective variety of dimension $n+1$ with $\rho(Y)=3$;
      \item $\pi\colon Y\to Z$ is a morphism with $\rho(Y/\PP^2)=1$ and $\dim Z=2$;
      \item $\mathcal{D}( \Delta)\simeq \PP^2(\R)*\mathbb{S}^{n-3}$.
      \end{enumerate}
The following construction generalizes \cite[Example 60]{KollarXu2016}. Consider the logCY pair $(Y', \Delta')$ defined by 
\begin{align*}
Y'& :=\PP^1_{[x_0:x_1]} \times \PP^1_{[y_0:y_1]}\times \PP^{n-1}_{[z_0:\ldots:z_{n-1}]}\\
\Delta_x & := \lbrace  x_0 x_1=0 \rbrace \subseteq \PP^1_{[x_0:x_1]}\\
\Delta_y & := \lbrace  y_0 y_1=0 \rbrace \subseteq \PP^1_{[y_0:y_1]}\\
\Delta_z & := \left\lbrace (z_0z_1+z_2^2+\ldots + z^2_{n-1})\prod^{n-1}_{i=2}z_i=0 \right\rbrace \subseteq \PP^{n-1}_{[z_0:\ldots:z_{n-1}]}\\
\Delta'& := \Delta_x \boxtimes \Delta_y \boxtimes \Delta_z = \left\lbrace x_0 x_1 y_0 y_1 (z_0z_1+z_2^2+\ldots + z^2_{n-1})\prod^{n-1}_{i=2}z_i=0\right\rbrace.
\end{align*}
The following facts hold.
\begin{enumerate}
\item The projection \[\pi'\colon Y' \to Z':=\PP^1 \times \PP^1\]  onto the first two factors of $Y'$ is a Mori fibre space.
\item The involution 
\begin{align*}
\tau\colon \PP^1 \times \PP^1\times \PP^{n-1} & \rightarrow \PP^1 \times \PP^1\times \PP^{n-1}\\
([x_0:x_1], [y_0:y_1], [z_0:z_1: \ldots: z_{n-1}]) & \mapsto ([x_1:x_0], [y_1:y_0], [z_1:z_0: \ldots: z_{n-1}])
\end{align*} 
preserves the boundary $\Delta'$ and descends to an involution onto $Z'$.
\item \label{diagramquotient} By construction, the following diagram commutes:
\[
 \xymatrix{
 (Y', \Delta') \ar@{->}[d]_{\pi'} \ar@{->}[r]^{q_Y \qquad \quad } & (Y, \Delta):=(Y'/\tau, q_{Y, *}\Delta') \ar@{->}[d]^{\pi}\\
Z' \ar@{->}[r]^{q_Z} & Z:=Z'/\tau.
 }
\] 
\item \label{quasietalequotientmap} The quotient map $q_Z$ and $q_Y$ are quasi-\'{e}tale. In particular, the two morphisms are \'{e}tale along $\Delta_x \boxtimes \Delta_y$ and its preimage via $\pi'$ respectively.
\end{enumerate} 
The pair $(Y, \Delta)$ is dlt logCY by (\ref{quasietalequotientmap}) and the morphism $\pi\colon Y \to Z$ is a Mori fibre space by the commutativity of the diagram (\ref{diagramquotient}).

 The dual complex $\mathcal{D}( \Delta')$ has combinatorial product type   
\begin{align*}
\mathcal{D}( \Delta')& = \mathcal{D}(\Delta_x) * \mathcal{D}(\Delta_y) * \mathcal{D}(\Delta_z)\\
& = \mathbb{S}^0 * \mathbb{S}^0 * \left( \sigma^{n-2}\cup_{\partial \sigma^{n-2}}\sigma^{n-2}\right)\\
& \simeq_{\text{PL}} \mathbb{S}^0 * \mathbb{S}^0 *\mathbb{S}^0 * \partial \sigma^{n-2}\\
& = \mathbb{S}^2 * \partial \sigma^{n-2}.
\end{align*}
Since the involution $\tau$ preserves $\Delta'$, it defines a PL-involution of the dual complex $\mathcal{D}( \Delta')$. It acts as the antipodal map on  $\mathbb{S}^2$ and it fixes $\partial \sigma^{n-2}$. We conclude that
\[\mathcal{D}( \Delta)\simeq \PP^2(\R)*\mathbb{S}^{n-3}.\]
\end{exam}
   \printbibliography
   \Addresses
\end{document}